\documentclass{amsart}
\usepackage{amsmath, amssymb, amsthm, mathrsfs, amscd}
\usepackage{graphicx}

\theoremstyle{plain} 
 \newtheorem{thm}{Theorem}[section]
 \newtheorem{lem}[thm]{Lemma}
 \newtheorem{cor}[thm]{Corollary}
 \newtheorem{prop}[thm]{Proposition}
 \newtheorem{claim}[thm]{Claim}

\theoremstyle{definition}
  \newtheorem{defn}[thm]{Definition}

\theoremstyle{remark}
  \newtheorem{rem}[thm]{Remark}

\newcommand{\stab}{{\rm Stab}}
\renewcommand{\mod}{{\rm Mod}}
\renewcommand{\pmod}{{\rm PMod}}
\newcommand{\aut}{{\rm Aut}}
\newcommand{\isom}{{\rm Isom}}
\renewcommand{\stab}{{\rm Stab}}
\newcommand{\cal}{\mathcal}

\newcommand{\ci}[2]{\cite[#1]{#2}}
\renewcommand{\c}{\curvearrowright}

\begin{document}

\title[Examples of amalgamated free products]{Examples of amalgamated free products and coupling rigidity}
\author{Yoshikata Kida}
\address{Department of Mathematics, Kyoto University, 606-8502 Kyoto, Japan}
\email{kida@math.kyoto-u.ac.jp}
\date{July 8, 2010, revised on October 5, 2011}

\begin{abstract}
We present amalgamated free products satisfying coupling rigidity with respect to the automorphism group of the associated Bass-Serre tree. 
As an application, we obtain orbit equivalence rigidity for amalgamated free products of mapping class groups.
\end{abstract}

\maketitle


\section{Introduction}

Measure equivalence is an equivalence relation between discrete countable groups and is defined so that two such groups are equivalent if there exists a measure space, called a coupling, equipped with certain measure-preserving actions of those groups (see Definition \ref{defn-me}). 
This is closely related to orbit equivalence between measure-preserving actions of discrete countable groups on probability measure spaces. 
We refer to \cite{furman-survey}, \cite{gab}, \cite{popa}, \cite{shalom-survey} and \cite{vaes} for recent developments and related topics on measure and orbit equivalence.

In \cite{kida-ama}, given two discrete countable groups $\Gamma_1$, $\Gamma_2$ satisfying rigidity in the sense of measure equivalence, we present a construction of an amalgamated free product $\Gamma =\Gamma_1\ast_A\Gamma_2$ satisfying rigidity of the same kind. 
The first step for the proof of this result is to show the coupling rigidity of $\Gamma$ with respect to the automorphism group of the Bass-Serre tree $T$ associated to the decomposition of $\Gamma$, which is proved on the assumption that $\Gamma_1$ and $\Gamma_2$ satisfy property (T). 
As a consequence, for any discrete countable group $\Lambda$ that is measure equivalent to $\Gamma$, we can find a simplicial action of $\Lambda$ on $T$ and obtain information on the structure of $\Lambda$ through the Bass-Serre theory. 
The purpose of the present paper is to establish coupling rigidity in the case where $\Gamma_1$ and $\Gamma_2$ do not necessarily satisfy property (T) (see Theorem \ref{thm-coup-rigid} for a precise statement). 
We apply it to proving rigidity results for amalgamated free products of mapping class groups. 
In this introduction, we state a result on orbit equivalence rigidity for such a group and give examples.

Let us introduce terminology to state Theorem \ref{thm-oer}. 
We mean by a discrete group a discrete and countable group. 
A Borel action of a discrete group $\Gamma$ on a measure space $(X, \mu)$ is said to be {\it f.f.m.p.}\ if $\mu$ is a finite positive measure on $X$ and if the action is essentially free and preserves $\mu$, where ``f.f.m.p." stands for ``essentially free and finite-measure-preserving". 
A measure-preserving action $\Gamma \c (X, \mu)$ is said to be {\it aperiodic} if any finite index subgroup of $\Gamma$ acts on $(X, \mu)$ ergodically. 
Two ergodic f.f.m.p.\ actions of discrete groups, $\Gamma \c (X, \mu)$ and $\Lambda \c (Y, \nu)$, are said to be {\it orbit equivalent (OE)} if there exists a Borel isomorphism $f$ between conull Borel subsets of $X$ and $Y$ such that $f_*\mu$ and $\nu$ are equivalent and the equality $f(\Gamma x)=\Lambda f(x)$ holds for a.e.\ $x\in X$.

Let $S$ be a connected, compact and orientable surface which may have non-empty boundary. 
Throughout the paper, we assume a surface to satisfy these conditions unless otherwise stated. 
Let $\mod^*(S)$ be the {\it extended mapping class group}, i.e., the group of isotopy classes of homeomorphisms from $S$ onto itself, where isotopy may move points in the boundary of $S$.
As shown in \cite{kida-oer} and \cite{kida-mer}, $\mod^*(S)$ satisfies rigidity in the sense of measure and orbit equivalence if $S$ is non-exceptional.

For a group $G$ and a subgroup $H$ of $G$, we set
\[{\rm LQN}_G(H)=\{\, g\in G\mid [H: gHg^{-1}\cap H]<\infty\,\}\]
and call it the {\it left-quasi-normalizer} of $H$ in $G$. 
We say that $H$ is {\it almost malnormal} in $G$ if $gHg^{-1}\cap H$ is finite for any $g\in G\setminus H$. 
We note that if $H$ is infinite and almost malnormal in $G$, then the equality ${\rm LQN}_G(H)=H$ holds.

\begin{thm}\label{thm-oer}
Let $S$ be a surface of genus $g$ with $p$ boundary components. 
We define an amalgamated free product $\Gamma =\Delta\ast_A\Delta$, where $\Delta$ and $A$ satisfy one of the following two conditions:
\begin{enumerate}
\item[(a)] Let $\Delta$ be a finite index subgroup of $\mod^*(S)$ with $3g+p\geq 5$ and $(g, p)\neq (1, 2), (2, 0)$.
Let $A$ be an infinite and almost malnormal subgroup of $\Delta$.
\item[(b)] Let $\Delta$ be a finite index subgroup of $\mod^*(S)$ with $3g+p\geq 7$ and $(g, p)\neq (2, 1)$.
Let $A$ be a subgroup of $\Delta$ such that the equality ${\rm LQN}_{\Delta}(A)=A$ holds and the group $\gamma A \gamma^{-1}\cap A$ is amenable for each $\gamma \in \Delta \setminus A$.
\end{enumerate}
Let $\Gamma \c (X, \mu)$ be an ergodic f.f.m.p.\ action such that the restriction $A\c (X, \mu)$ is aperiodic. 
If the action $\Gamma \c (X, \mu)$ is OE to an ergodic f.f.m.p.\ action $\Lambda \c (Y, \nu)$ of a discrete group $\Lambda$, then those two actions are indeed conjugate.
\end{thm}

We obtain this theorem as a direct consequence of the result that the cocycle $\alpha \colon \Gamma \times X\rightarrow \Lambda$ associated with the OE is cohomologous to the cocycle arising from an isomorphism from $\Gamma$ onto $\Lambda$ (see Corollary \ref{cor-oer} for a precise statement).

Let us present three kinds of examples of $\Delta$ and $A$ satisfying the assumption in Theorem \ref{thm-oer}. 
The groups $\Delta$ and $A$ in the first example fulfill condition (a) in Theorem \ref{thm-oer}, and those in the second and third examples fulfill condition (b).
\begin{itemize}
\item Let $\Delta$ be a finite index subgroup of $\mod^*(S)$ with $3g+p\geq 5$ and $(g, p)\neq (1, 2), (2, 0)$.
Define $A$ as the stabilizer in $\Delta$ of the pair of pseudo-Anosov foliations for a pseudo-Anosov element.
\item Let $\Delta$ be a finite index subgroup of $\mod^*(S)$ with $3g+p\geq 7$ and $(g, p)\neq (2, 1)$.
Define $A$ as the stabilizer in $\Delta$ of a pants decomposition of $S$.
\item Let $\Delta$ be a finite index subgroup of $\mod^*(S)$ with $g\geq 3$ and $p=0$.
Define $A$ as the stabilizer in $\Delta$ of a Teichm\"uller disk so that $A$ contains a pseudo-Anosov element.
\end{itemize}

This paper is organized as follows. 
In Section \ref{sec-group}, we review basic facts on discrete measured groupoids and their amenability.
In Section \ref{sec-b-ame}, given an action of a discrete group on a tree, we present a sufficient condition for the action on the boundary of the tree to be amenable in a measure-theoretic sense. 
This is applied to the action on the Bass-Serre tree for an amalgamated free product of two discrete groups in Classes $\mathscr{F}$ and $\mathscr{A}$, introduced in Sections \ref{sec-c-f} and \ref{sec-c-a}, respectively. 
Basic properties and examples of groups in those classes are also discussed. 
In particular, mapping class groups of certain surfaces are shown to belong to these classes. 
In Section \ref{sec-coup-rigid}, we prove coupling rigidity of an amalgamated free product $\Gamma$ of two discrete groups in Classes $\mathscr{F}$ and $\mathscr{A}$. 
Applying results in \cite{kida-ama}, we show orbit equivalence rigidity of an ergodic f.f.m.p.\ action of $\Gamma$. 
The aforementioned three examples of subgroups of mapping class groups are discussed in Section \ref{sec-mcg}.


\section{Discrete measured groupoids}\label{sec-group}

In this section, we collect basic facts on discrete measured groupoids. 
We recommend the reader to consult \cite{kechris} and Chapter XIII, \S 3 in \cite{take3} for basic knowledge of standard Borel spaces and discrete measured groupoids, respectively.

\subsection{Terminology}

We mean by a {\it standard finite measure space} a standard Borel space with a finite positive measure. 
When the measure is a probability one, we refer to it as a {\it standard probability space}. 
For a discrete measured groupoid $\mathcal{G}$ on a standard finite measure space $(X, \mu)$ and a Borel subset $A$ of $X$ with positive measure, we denote by
\[(\mathcal{G})_A=\{\, g\in \mathcal{G}\mid r(g), s(g)\in A\, \}\label{restriction}\]
the groupoid restricted to $A$, where $r, s\colon \mathcal{G}\rightarrow X$ are the range and source maps, respectively. 
If $A$ is a Borel subset of $X$, then $\mathcal{G}A$ stands for the saturation
\[\mathcal{G}A=\{\, r(g)\in X\mid g\in \mathcal{G}, s(g)\in A\, \}\label{saturation},\]
which is a Borel subset of $X$. 
We say that a discrete measured groupoid $\mathcal{G}$ on a standard finite measure space $(X, \mu)$ is {\it finite} if for a.e.\ $x\in X$, $r^{-1}(x)$ consists of at most finitely many points. 
We say that $\mathcal{G}$ is of {\it infinite type} if for a.e.\ $x\in X$, $r^{-1}(x)$ consists of infinitely many points.

Given a Borel action of a discrete group $\Gamma$ on a standard Borel space $X$, we say that a $\sigma$-finite positive measure $\mu$ on $X$ is {\it quasi-invariant} for that action if it preserves the class of $\mu$. 
In this case, the action $\Gamma \c (X, \mu)$ is said to be {\it non-singular}.
 We denote by $\Gamma \ltimes (X, \mu)$ (or $\Gamma \ltimes X$ if there is no confusion) the discrete measured groupoid associated with that action.
This is equal to $\Gamma \times X$ as a set, and the range and source maps $r, s\colon \Gamma \times X\to X$ are defined by $r(\gamma, x)=\gamma x$ and $s(\gamma, x)=x$ for $\gamma \in \Gamma$ and $x\in X$.
The product is defined by $(\gamma_1, \gamma_2x)(\gamma_2, x)=(\gamma_1\gamma_2, x)$ for $\gamma_1, \gamma_2\in \Gamma$ and $x\in X$.
For $x\in X$, the unit at $x$ is defined to be $(e, x)$, where $e$ is the neutral element of $\Gamma$.

There is a close connection between orbit equivalence and isomorphism of discrete measured groupoids associated with group actions. 
Let $\Gamma \c (X, \mu)$ and $\Lambda \c (Y, \nu)$ be ergodic f.f.m.p.\ actions on standard finite measure spaces, and let $\mathcal{G}$ and $\mathcal{H}$ be the associated groupoids, respectively. 
It follows that $\mathcal{G}$ and $\mathcal{H}$ are isomorphic as discrete measured groupoids if and only if the two actions are OE. 
If there are Borel subsets $A\subset X$ and $B\subset Y$ with positive measure such that $(\mathcal{G})_A$ and $(\mathcal{H})_{B}$ are isomorphic, then the two actions are said to be {\it weakly orbit equivalent (WOE)}.

Let us say that a discrete measured groupoid $\cal{G}$ on a standard finite measure space $(X, \mu)$ is {\it nowhere amenable} if for any Borel subset $A$ of $X$ with positive measure, the restriction $(\cal{G})_A$ is non-amenable. 
We refer to \cite{ar-book} for amenability of discrete measured groupoids.

\subsection{Amenable actions}\label{subsec-ame}

We shall recall a definition of amenable actions in a measure-theoretic sense, originally introduced in \cite{zim-ame-action} in a different way. 
For a discrete group $\Gamma$, we mean by a {\it $\Gamma$-space} a countable discrete set on which $\Gamma$ acts by bijections.

\begin{defn}
Let $\Gamma$ be a discrete group and $V$ a $\Gamma$-space. Suppose that $\Gamma$ acts on a standard probability space $(X, \mu)$ non-singularly. 
An {\it invariant measurable system of means} for the $\Gamma$-space $V$ over $(X, \mu)$ is a family of means $\{ m_x\}_{x\in X}$ on $\ell^{\infty}(V)$ such that for each $\varphi \in L^{\infty}(X\times V)$,
\begin{itemize}
\item the function $X\ni x\mapsto m_x(\varphi(x, \cdot))$ is $\mu$-measurable; and 
\item the equality $m_{\gamma^{-1}x}(\varphi(\gamma^{-1}x, \cdot))=m_x((\gamma \varphi)(x, \cdot))$ holds for any $\gamma \in \Gamma$ and $\mu$-a.e.\ $x\in X$,  
\end{itemize}
where the action $\Gamma \c L^{\infty}(X\times V)$ is defined by $(\gamma \varphi)(x, v)=\varphi(\gamma^{-1}x, \gamma^{-1}v)$ for any $\gamma \in \Gamma$, $\varphi \in L^{\infty}(X\times V)$ and $(x, v)\in X\times V$.
\end{defn}

\begin{defn}
Let $\Gamma \c X$ be a Borel action of a discrete group $\Gamma$ on a standard Borel space $X$.
\begin{enumerate}
\item[(i)] Let $\mu$ be a probability measure on $X$ quasi-invariant for the action of $\Gamma$. 
We say that the action $\Gamma \c (X, \mu)$ is {\it amenable} if there exists an invariant measurable system of means for the $\Gamma$-space $\Gamma$ over $(X, \mu)$, where $\Gamma$ acts on $\Gamma$ by left multiplication.
\item[(ii)] We say that the action $\Gamma \c X$ is {\it amenable in a measure-theoretic sense} if the action $\Gamma \c (X, \mu)$ is amenable for any probability measure $\mu$ on $X$ quasi-invariant for the action $\Gamma \c X$.
\end{enumerate}
\end{defn}

We note that a non-singular action $\Gamma \c (X, \mu)$ is amenable in the above sense if and only if the associated groupoid is amenable. 
The following are useful criteria for actions or groupoids to be amenable. 
In this paper, we mean by a subgroupoid of $\mathcal{G}$ a Borel subgroupoid of $\mathcal{G}$ whose unit space is the same as that of $\mathcal{G}$.

\begin{prop}[\ci{Proposition A.7}{kida-mcg}]\label{prop-ext-ame}
Let $\Gamma$ be a discrete group acting on a standard probability space $(X, \mu)$ non-singularly, and let $\Lambda$ be a subgroup of $\Gamma$. 
If there is an invariant measurable system of means for the $\Gamma$-space $\Gamma/\Lambda$ defined by left multiplication, over $(X, \mu)$, and if the action $\Lambda \c (X, \mu)$ is amenable, then the action $\Gamma \c (X, \mu)$ is also amenable.
\end{prop}

\begin{thm}[\ci{Corollary C}{aeg}]\label{thm-ext-ame}
Let $\Gamma$ be a discrete group acting on two standard probability spaces $(X, \mu)$ and $(Y, \nu)$ non-singularly. 
Suppose that there is a $\Gamma$-equivariant Borel map $f\colon X\rightarrow Y$ with $f_*\mu =\nu$. 
If the action $\Gamma \c (Y, \nu)$ is amenable, then the action $\Gamma \c (X, \mu)$ is also amenable.
\end{thm}

\begin{prop}\label{prop-ame-subgrd}
Suppose that a discrete group $\Gamma$ acts on a standard Borel space $B$ and acts on a standard probability space $(X, \mu)$ non-singularly. 
Let $\mathcal{S}$ be a subgroupoid of the associated groupoid $\mathcal{G}=\Gamma \ltimes (X, \mu)$. 
If the action $\Gamma \c B$ is amenable in a measure-theoretic sense and if there is a Borel map $\varphi \colon X\rightarrow B$ satisfying the equality $\varphi(\gamma x)=\gamma \varphi(x)$ for a.e.\ $(\gamma, x)\in \mathcal{S}$, then $\mathcal{S}$ is amenable.
\end{prop}

The last proposition can be obtained by following the proof of Theorem 5.10 in \cite{kida-survey} after exchanging symbols appropriately. 
We make use of amenability to find a fixed point when a groupoid acts on a space. 
Such a fixed point is formulated as follows.

\begin{defn}\label{defn-inv-map}
Let $\mathcal{G}$ be a discrete measured groupoid on a standard probability space $(X, \mu)$, and let $S$ be a standard Borel space. 
Suppose that we have a Borel action of a discrete group $\Gamma$ on $S$ and a groupoid homomorphism $\rho \colon \mathcal{G}\rightarrow \Gamma$. 
A Borel map $\varphi \colon X\rightarrow S$ is said to be {\it $(\cal{G}, \rho)$-invariant} if it satisfies the equality
\[\rho(g)\varphi(s(g))=\varphi(r(g))\quad \textrm{for\ a.e.}\ g\in \mathcal{G}.\]
When $\rho$ is understood from the context, this map $\varphi$ is said to be {\it $\cal{G}$-invariant}.
More generally, if $A$ is a Borel subset of $X$ with positive measure and if a Borel map $\varphi \colon A\rightarrow S$ satisfies the above equality for a.e.\ $g\in (\mathcal{G})_A$, then we also say that $\varphi$ is {\it $\cal{G}$-invariant}.  
\end{defn}

We will often apply amenability of actions or groupoids in the context of the following proposition, which directly follows from the formulation of amenability in terms of the fixed point property (see \cite{zim-ame-action} and Theorem 4.2.7 in \cite{ar-book}).

\begin{prop}\label{prop-ame-basic}
Let $\mathcal{G}$ be a discrete measured groupoid on a standard probability space $(X, \mu)$, $\Gamma$ a discrete group, and $\rho \colon \mathcal{G}\rightarrow \Gamma$ a groupoid homomorphism. 
Suppose that $\Gamma$ acts on a compact metrizable space $K$ continuously. 
We denote by $M(K)$ the space of probability measures on $K$, on which $\Gamma$ naturally acts. 
If $\mathcal{G}$ is amenable, then there exists a $\cal{G}$-invariant Borel map from $X$ into $M(K)$.
\end{prop}

Finally, we present an observation on invariant Borel maps, which will be used repeatedly in Sections \ref{sec-c-f} and \ref{sec-c-a}.

\begin{lem}[\ci{Lemma 2.3}{kida-ama}]\label{lem-inv-ext}
Let $\mathcal{G}$ be a discrete measured groupoid on a standard probability space $(X, \mu)$. 
Suppose that we have a Borel action of a discrete group $\Gamma$ on a standard Borel space $S$ and a groupoid homomorphism $\rho \colon \mathcal{G}\rightarrow \Gamma$. 
If $A$ is a Borel subset of $X$ with positive measure and if $\varphi \colon A\rightarrow S$ is a $\cal{G}$-invariant Borel map, then $\varphi$ extends to a $\cal{G}$-invariant Borel map from $\cal{G}A$ into $S$. 
\end{lem}


\subsection{Normal subgroupoids}\label{subsec-nor}

In Section 2.4 of \cite{kida-mer}, we introduced normal subgroupoids of a discrete measured groupoid, based on \cite{fsz} and Section 4.6.1 in \cite{kida-mcg}. 
This is a generalization of normal subrelations of a discrete measured equivalence relation and also a generalization of normal subgroups of a discrete group.

Let $\mathcal{G}$ be a discrete measured groupoid on a standard probability space $(X, \mu)$ and $r, s\colon \mathcal{G}\rightarrow X$ the range and source maps, respectively. 
Let $\mathcal{S}$ be a subgroupoid of $\mathcal{G}$. 
We denote by ${\rm End}_{\mathcal{G}}(\mathcal{S})$ the set of all Borel maps $\varphi \colon {\rm dom}(\varphi)\rightarrow \mathcal{G}$, where ${\rm dom}(\varphi)$ is a Borel subset of $X$, such that 
\begin{enumerate}
\item[(a)] for a.e.\ $x\in {\rm dom}(\varphi)$, the equality $s(\varphi(x))=x$ holds; and
\item[(b)] for a.e.\ $g\in (\mathcal{G})_{{\rm dom}(\varphi)}$, we have
\[g\in \mathcal{S}\quad \textrm{if\ and\ only\ if}\quad \varphi(r(g))g \varphi(s(g))^{-1}\in \mathcal{S}.\]
\end{enumerate}

\begin{defn}
Let $\cal{G}$ be a discrete measured groupoid on a standard probability space $(X, \mu)$.
A subgroupoid $\mathcal{S}$ of $\mathcal{G}$ is said to be {\it normal} in $\mathcal{G}$ if there exists a collection of countably many maps in ${\rm End}_{\mathcal{G}}(\mathcal{S})$, $\{ \phi_n\}$, such that for a.e.\ $g\in \mathcal{G}$, we have $r(g)\in {\rm dom}(\phi_n)$ and $\phi_n(r(g))g\in \mathcal{S}$ for some $n$. 
In this case, we write $\mathcal{S}\lhd \mathcal{G}$.
\end{defn}

The following basic properties are immediately verified. 
We refer to Lemma 2.16 in \cite{kida-mer} for a proof of assertion (ii).

\begin{lem}
The following assertions hold:
\begin{enumerate}
\item Let $G$ be a discrete group acting on a standard probability space $(X, \mu)$ non-singularly, and let $H$ be a normal subgroup of $G$. 
Then $H\ltimes X$ is a normal subgroupoid of $G\ltimes X$.
\item Let $\mathcal{G}$ be a discrete measured  groupoid on a standard probability space $(X, \mu)$ and $\mathcal{S}$ a normal subgroupoid of $\mathcal{G}$. If $A$ is a Borel subset of $X$ with positive measure, then $(\mathcal{S})_A$ is normal in $(\mathcal{G})_A$.
\end{enumerate}
\end{lem}


\section{Boundary amenability of actions on trees}\label{sec-b-ame}

Given an amalgamated free product $\Gamma =\Gamma_1\ast_A\Gamma_2$ of discrete groups, one associates a simplicial tree $T$, called the Bass-Serre tree, on which $\Gamma$ acts by simplicial automorphisms. 
Theorem \ref{thm-ame-action} stated below gives a sufficient condition for the action of $\Gamma$ on the boundary $\partial T$ of $T$ to be amenable in a measure-theoretic sense.

Let $T$ be a connected simplicial tree which has at most countably many simplices and is not necessarily locally finite. 
We denote by $V(T)$ the set of vertices of $T$. 
Let $d$ be a path metric on $T$ so that each edge of $T$ is isometric to the unit interval $[0, 1]$ with the Euclidean metric. 
An ideal boundary $\partial T$\label{boundary} of $T$ is defined as the set of all equivalence classes of infinite geodesic rays in $T$, where two geodesic rays in $T$ are equivalent if the Hausdorff distance between them is finite. 
For any two distinct points $x$, $y$ in the union $\Delta T=V(T)\cup \partial T$\label{delta-t}, there exists a unique geodesic path from $x$ to $y$. 
We denote it by $l_x^y$\label{geodesic} and often identify $l_x^y$ with the sequence of vertices that $l_x^y$ passes through. 
For $x\in \Delta T$ and a subset $A$ of $V(T)$, we define the set
\[M(x, A)=\{\, y\in \Delta T\mid l_x^y\cap (A\setminus \{ x\})=\emptyset \, \}.\] 
The following theorem defines a topology on $\Delta T$ and gives basic properties of it.

\begin{thm}[\ci{Section 8}{bow-rel-hyp}]
The collection of the sets $M(x, A)$ for all $x\in \Delta T$ and all finite subsets $A$ of $V(T)$ defines an open basis for a topology on $\Delta T$ satisfying the following properties:
\begin{enumerate}
\item[(i)] The space $\Delta T$ is compact and Hausdorff. 
\item[(ii)] The set $V(T)$ is dense in $\Delta T$, and the boundary $\partial T$ is a $G_{\delta}$-subset and in particular a Borel subset of $\Delta T$.
\item[(iii)] Any simplicial automorphism of $T$ is uniquely extended to a homeomorphism from $\Delta T$ onto itself.
\item[(iv)] The relative topology on $\partial T$ coincides with the topology on it as the Gromov boundary of $T$.
\end{enumerate}
\end{thm}

It is known that for each word-hyperbolic group, its action on its boundary is amenable in a measure-theoretic sense. 
This fact is originally proved by Adams \cite{adams-ame}. 
The construction of the map $H$ in the proof of the following theorem basically relies on Germain's proof of this fact in Appendix B of \cite{ar-book}.

\begin{thm}\label{thm-ame-action}
Let $T$ be a connected simplicial tree with at most countably many simplices and $d$ the metric on $T$ defined above. 
Let $\Gamma$ be a discrete group acting on $T$ by simplicial automorphisms. 
Suppose that the following three conditions hold:
\begin{enumerate}
\item[(a)] The boundary $\partial T$ is non-empty.
\item[(b)] The action of $\Gamma$ on $V(T)$ has only finitely many orbits.
\item[(c)] There is a positive integer $N$ such that for any two vertices $u, v\in V(T)$ with $d(u, v)\geq N$, their stabilizer in $\Gamma$ defined as
\[\{\, \gamma \in \Gamma \mid \gamma u=u,\ \gamma v=v\, \}\]
is amenable.
\end{enumerate}
Then the action $\Gamma \c \partial T$ is amenable in a measure-theoretic sense. 
\end{thm}

\begin{proof}
This proof follows the author's argument in Theorem 3.19 of \cite{kida-mcg}, where the action of the mapping class group on the boundary of the complex of curves is dealt with. 
To prove amenability of that action, we constructed a certain map $H$ associated to a $\delta$-hyperbolic graph satisfying properties (F1) and (F2), which were originally introduced by Bowditch \cite{bow-tight}. 
It is easy to check these properties for our tree $T$. 
(In the notation right before Lemma 2.8 in \cite{kida-mcg}, it is enough to put $\delta=\delta_0=0$, $\delta_1=1$, $\mathcal{L}_{T}(x, y)=\{l_x^y\}$ and $P_0=P_1=1$.) 
As discussed in Theorem 2.9 and Remark 2.2 in \cite{kida-mcg}, we can construct a map
\[H\colon \partial T\times V(T)\times \mathbb{N}\rightarrow \ell^1(V(T))\] 
satisfying the following properties, where $\ell^1(V(T))$ is the Banach space of all $\ell^1$-functions on $V(T)$:
\begin{itemize}
\item For any $v, u\in V(T)$ and any $n\in \mathbb{N}$, the function $H(\cdot , v, n)(u)$ on $\partial T$ is Borel.
\item For each $(x, v, n)\in \partial T\times V(T)\times \mathbb{N}$, the function $H(x, v, n)$ on $V(T)$ is non-negative and satisfies $\Vert H(x, v, n)\Vert_1=1$.
\item For each number $R>0$, the value $\Vert H(x, v, n)-H(x, u, n)\Vert_1$ converges to $0$ as $n\rightarrow \infty$, uniformly for $x\in \partial T$ and $v, u\in V(T)$ with $d(v, u)<R$.
\item The equality $H(\gamma x, \gamma v, n)=\gamma H(x, v, n)$ holds for any $(x, v, n)\in \partial T\times V(T)\times \mathbb{N}$ and any $\gamma \in \Gamma$, where the action $\Gamma \c \ell^1(V(T))$ is defined by $(\gamma \varphi)(v)=\varphi(\gamma^{-1}v)$ for any $\gamma \in \Gamma$, $\varphi \in \ell^1(V(T))$ and $v\in V(T)$.
\end{itemize}

Pick a probability measure $\mu$ on $\partial T$ quasi-invariant for the action $\Gamma \c \partial T$. 
Fix $v_0\in V(T)$.
For each $n\in \mathbb{N}$, we define a function $f_n$ on $\partial T\times V(T)$ by setting
\[f_n(x, v)=H(x, v_0, n)(v)\quad \textrm{for\ each}\ (x, v)\in \partial T\times V(T).\] 
We then have $\sum_{v\in V(T)}f_n(x, v)=1$ for any $x\in \partial T$ and
\[\lim_{n\rightarrow \infty}\sup_{x\in \partial T}\sum_{v\in V(T)}\vert f_n(x, v)-f_n(\gamma^{-1}x, \gamma^{-1}v)\vert =0\]
for any $\gamma \in \Gamma$. It follows from Remark 3.2.6 in \cite{ar-book} (see also Theorem A.2 in \cite{kida-mcg}) that there exists an invariant measurable system $\{ m_x\}_{x\in \partial T}$ of means for the $\Gamma$-space $V(T)$ over $(\partial T, \mu)$. 
Recall that $m_x$ is a mean on $\ell^{\infty}(V(T))$ for each $x\in \partial T$.

Let $\mathcal{V}\subset V(T)$ be a set of representatives of orbits for the action $\Gamma \c V(T)$. Since $\mathcal{V}$ is finite by condition (b), the equality
\[1=m_x(1)=m_x\left(\sum_{v\in \mathcal{V}}\chi_{\mathcal{O}(v)}\right)=\sum_{v\in \mathcal{V}}m_x(\chi_{\mathcal{O}(v)})\]
holds for $\mu$-a.e.\ $x\in \partial T$, where $\mathcal{O}(v)$ is the $\Gamma$-orbit in $V(T)$ containing $v$. 
We then have $\partial T=\bigcup_{v\in \mathcal{V}}X_v$, where for each $v\in \mathcal{V}$, we define a $\Gamma$-invariant Borel subset $X_v$ of $\partial T$ as
\[X_v=\{\, x\in \partial T\mid m_x(\chi_{\mathcal{O}(v)})\geq 1/\left|\mathcal{V}\right|\, \}.\]
To prove the theorem, it is enough to verify that the action of $\Gamma$ on $X_v$ equipped with the restriction of $\mu$ is amenable for each $v\in \mathcal{V}$ with $\mu(X_v)>0$. 
By normalization, we get an invariant measurable system of means for the $\Gamma$-space $\mathcal{O}(v)$ over $X_v$. 
Let $\Gamma_v$ denote the stabilizer of $v$ in $\Gamma$.
Since $\mathcal{O}(v)$ is identified with $\Gamma /\Gamma_v$ as a $\Gamma$-space, the theorem is a consequence of Proposition \ref{prop-ext-ame} and the following lemma, where condition (c) will be used.
\end{proof}

\begin{lem}
In the notation of Theorem \ref{thm-ame-action}, the action $\Gamma_v\c \partial T$ is amenable in a measure-theoretic sense for any $v\in V(T)$, where $\Gamma_v$ is the stabilizer of $v$ in $\Gamma$.
\end{lem}

\begin{proof}
Fix $v\in V(T)$ and let $\mu$ be a probability measure on $\partial T$ quasi-invariant for the action $\Gamma_v\c \partial T$. 
We put
\[S=\{\, u\in V(T)\mid d(u, v)=N\, \},\]
where $N$ is the number in condition (c) of Theorem \ref{thm-ame-action}. 
The set $S$ is then $\Gamma_v$-invariant. 
For each $x\in \partial T$, let $F(x)$ be the unique point in the intersection of $S$ and the geodesic from $v$ to $x$. 
The map $F\colon \partial T\rightarrow S$ is then Borel and $\Gamma_v$-equivariant. 
Let $\mathcal{S}\subset S$ be a set of representatives of orbits for the action $\Gamma_v\c S$. 
For each $u\in \mathcal{S}$, we put
\[\Gamma_{vu}=\{\, \gamma \in \Gamma_v\mid \gamma u=u\, \},\quad Y_u=F^{-1}(\mathcal{O}_v(u)), \]
where $\mathcal{O}_v(u)$ is the $\Gamma_v$-orbit in $S$ containing $u$. The set $Y_u$ is a $\Gamma_v$-invariant Borel subset of $\partial T$ for each $u\in \mathcal{S}$, and the equality $\partial T=\bigsqcup_{u\in \mathcal{S}}Y_u$ holds. 
To prove the lemma, it is enough to show that the action $\Gamma_v\c (Y_u, \mu|_{Y_u})$ is amenable for each $u\in \mathcal{S}$ with $\mu(Y_u)>0$.

We claim that there is an invariant measurable system of means for the $\Gamma_v$-space $\Gamma_v/\Gamma_{vu}$ over $(Y_u, \mu|_{Y_u})$ for each $u\in \mathcal{S}$ with $\mu(Y_u)>0$. 
The lemma follows from this claim and Proposition \ref{prop-ext-ame} because $\Gamma_{vu}$ is amenable by condition (c) in Theorem \ref{thm-ame-action} and the action $\Gamma_{vu}\c Y_u$ is thus amenable. 
Note that $\Gamma_v/\Gamma_{vu}$ is identified with $\mathcal{O}_v(u)$ as a $\Gamma_v$-space. 
For each $x\in Y_u$, we define a mean $m_x$ on $\ell^{\infty}(\mathcal{O}_v(u))$ by the formula
\[m_x(\varphi)=\varphi(F(x))\quad \textrm{for\ each}\ \varphi \in \ell^{\infty}(\mathcal{O}_v(u)).\]
The family $\{ m_x\}_{x\in Y_u}$ is then an invariant measurable system of means for the $\Gamma_v$-space $\mathcal{O}_v(u)$ over $(Y_u, \mu|_{Y_u})$.
\end{proof}

\begin{cor}\label{cor-ame-action}
In the notation of Theorem \ref{thm-ame-action}, let $\partial_2T$ be the quotient space of $\partial T\times \partial T$ by the action of the symmetric group of two letters defined by exchanging the coordinates. 
Then the action $\Gamma \c \partial_2T$ induced by the diagonal action $\Gamma \c \partial T\times \partial T$ is amenable in a measure-theoretic sense.
\end{cor}

This is a direct consequence of Theorem \ref{thm-ame-action}. 
We refer to Lemma 5.2 in \cite{adams} or Lemma 4.32 in \cite{kida-mcg} for a precise proof.


\section{Class $\mathscr{F}$}\label{sec-c-f}

In this section, we introduce Class $\mathscr{F}$ of discrete groups and discuss its basic properties.
We also present examples of groups in that class. 
Let $T$ be a connected simplicial tree with at most countably many simplices. 
In what follows, we mean by a tree one satisfying these conditions unless otherwise stated. 
Let us denote by $\aut^*(T)$ the simplicial automorphism group of $T$ equipped with the standard Borel structure arising from the pointwise convergence topology. 
When $\partial T$ is non-empty, we define $\partial_2T$ to be the quotient space of $\partial T\times \partial T$ by the action of the symmetric group of two letters defined by exchanging the coordinates. 
We denote by $V(T)$ and $S(T)$ the sets of vertices and simplices of $T$, respectively.
Unless there is a confusion, $V(T)$ is often identified with the set of 0-simplices of $T$.

\begin{defn}\label{defn-c-f}
We say that a discrete group $\Gamma$ belongs to Class $\mathscr{F}$ if $\Gamma$ is infinite and the following statement holds: Let $\Gamma \c (X, \mu)$ be a measure-preserving action on a standard probability space and $\mathcal{G}$ the associated groupoid.
Suppose that we have a tree $T$ with $\partial T\neq \emptyset$ and a Borel cocycle $\rho \colon \Gamma \times X\rightarrow \aut^*(T)$ satisfying the following three conditions:
\begin{enumerate}
\item[(I)] For any three mutually distinct vertices $v_1, v_2, v_3\in V(T)$, the subgroupoid of $\cal{G}$ defined as
\[\{\, (\gamma, x)\in \mathcal{G}\mid \rho(\gamma, x)v_i=v_i \ \textrm{for\ any}\ i=1, 2, 3\, \}\]
is finite.
\item[(II)] For any Borel map $\phi \colon X\rightarrow \partial_2T$, the subgroupoid of $\cal{G}$ defined as
\[\{\, (\gamma, x)\in \mathcal{G}\mid \phi(\gamma x)=\rho(\gamma, x)\phi(x)\, \}\]
is amenable.
\item[(III)] For any $\gamma \in \Gamma$ and a.e.\ $x\in X$, the automorphism $\rho(\gamma, x)$ of $T$ has no inversion.
\end{enumerate}
Then there exists a $(\mathcal{G}, \rho)$-invariant Borel map $\varphi \colon X\rightarrow V(T)$.
\end{defn}

We now explain a motivation for introducing Class $\mathscr{F}$. 
Let $\Gamma =\Gamma_1\ast_A\Gamma_2$ be an amalgamated free product of discrete groups, and let $T$ be the Bass-Serre tree associated to the decomposition of $\Gamma$.
Suppose that we have a self-coupling of $\Gamma$ and the ME cocycle $\alpha \colon \Gamma \times X\rightarrow \Gamma$ associated with it, where $X$ is a fundamental domain for the action of $\{ e\} \times \Gamma$ on that coupling (see Section \ref{sec-coup-rigid} for these terms). 
It is shown that if $\Gamma_1$ and $\Gamma_2$ belong to $\mathscr{F}$ and if $A$ is almost malnormal in both $\Gamma_1$ and $\Gamma_2$, then for each $i=1, 2$, there exists a Borel map $\varphi_i\colon X\rightarrow V(T)$ with
\[\alpha(\gamma, x)\varphi_i(x)=\varphi_i(\gamma \cdot x)\quad \textrm{for\ any}\ \gamma \in \Gamma_i \textrm{\ and\ a.e.\ }x\in X.\] 
These maps $\varphi_1$ and $\varphi_2$ play an important role in the proof of coupling rigidity of $\Gamma$ with respect to $\aut^*(T)$, discussed in Section \ref{sec-coup-rigid}.

\subsection{Basic properties of Class $\mathscr{F}$}\label{subsec-c-f-def}

We show a few consequences of the existence of the $(\cal{G}, \rho)$-invariant Borel map $\varphi$ in Definition \ref{defn-c-f} and give hereditary properties of groups in Class $\mathscr{F}$. 
If there is such an invariant Borel map into $V(T)$, then one can construct a canonical invariant Borel map valued in $S(T)$ as follows.

\begin{lem}\label{lem-inv-s}
Let $\Gamma \c (X, \mu)$ be a measure-preserving action of a discrete group on a standard probability space and $\mathcal{G}$ the associated groupoid. 
Suppose that we have
\begin{itemize}
\item a tree $T$ with $\partial T\neq \emptyset$ and a Borel cocycle $\rho \colon \Gamma \times X\rightarrow \aut^*(T)$ satisfying condition (I) in Definition \ref{defn-c-f}; and
\item a Borel subset $A$ of $X$ with positive measure and a subgroupoid $\mathcal{S}$ of $(\mathcal{G})_A$ of infinite type.
\end{itemize}
If there exists an $\cal{S}$-invariant Borel map from $A$ into $V(T)$, then the following assertions hold:
\begin{enumerate}
\item There exists an essentially unique, $\cal{S}$-invariant Borel map $\varphi_0\colon A\rightarrow S(T)$ such that for any Borel subset $B\subset A$ with positive measure and any $\cal{S}$-invariant Borel map $\varphi \colon B\rightarrow V(T)$, the inclusion $\varphi(x)\subset \varphi_0(x)$ holds for a.e.\ $x\in B$.
\item If $\cal{T}$ is a subgroupoid of $(\mathcal{G})_A$ with $\mathcal{S}\lhd \mathcal{T}$, then the map $\varphi_0$ in assertion (i) is $\cal{T}$-invariant.
\end{enumerate}
\end{lem}

\begin{proof}
The idea of the following construction of $\varphi_0$ appears in Lemma 3.2 of \cite{adams}. 
Let $I$ be the set of all $\cal{S}$-invariant Borel maps from $A$ into $S(T)$. 
By assumption, $I$ is non-empty. 
For each $\varphi \in I$, we set
\[S_{\varphi}=\{\, x\in A\mid \varphi(x)\in S(T)\setminus V(T)\,\}.\]
One can find $\varphi_0\in I$ with $\mu(S_{\varphi_0})=\sup_{\varphi \in I}\mu(S_{\varphi})$. It follows from condition (I) for $\rho$ in Definition \ref{defn-c-f} that this $\varphi_0$ satisfies the desired property in assertion (i).

We give only a sketch of the proof of assertion (ii) because this is a verbatim translation of the proof of Lemma 6.7 in \cite{kida-survey} after exchanging symbols appropriately.
Let us denote by $[[\cal{T}]]_{\cal{S}}$ the set of all maps $\phi$ in ${\rm End}_{\cal{T}}(\cal{S})$ such that the composition $r\circ \phi \colon {\rm dom}(\phi)\rightarrow X$ is injective on a conull Borel subset of ${\rm dom}(\phi)$, where $r\colon \cal{G}\rightarrow X$ is the range map for $\cal{G}$.
For each $\phi \in [[\cal{T}]]_{\cal{S}}$, one can prove that the map
\[{\rm dom}(\phi)\ni x\mapsto \rho(\phi(x)^{-1})\varphi_0(r(\phi(x)))\in S(T)\]
is $\mathcal{S}$-invariant.
It follows from the property of $\varphi_0$ stated in assertion (i) that the inclusion $\rho(\phi(x)^{-1})\varphi_0(r(\phi(x)))\subset \varphi_0(x)$ holds for a.e.\ $x\in {\rm dom}(\phi)$.
By considering the ``inverse'' of $\phi$, i.e., the map in $[[\mathcal{T}]]_{\mathcal{S}}$ defined by
\[r\circ \phi({\rm dom}(\phi))\ni y\mapsto \phi((r\circ \phi)^{-1}(y))^{-1}\in \mathcal{T},\]
one can prove the converse inclusion. 
The reader should consult the cited reference for more details.
Since $\mathcal{T}$ is generated by $[[\mathcal{T}]]_{\mathcal{S}}$, the map $\varphi_0$ is $\mathcal{T}$-invariant.
\end{proof}

\begin{prop}\label{prop-basic-c}
Class $\mathscr{F}$ satisfies the following properties:
\begin{enumerate}
\item Let $\Gamma$ and $\Lambda$ be discrete groups with $\Lambda <\Gamma$ and $[\Gamma :\Lambda]<\infty$. 
Then $\Gamma \in \mathscr{F}$ if and only if $\Lambda \in \mathscr{F}$.
\item If $\Gamma$ and $\Lambda$ are discrete groups with $\Lambda \lhd \Gamma$ and $\Lambda \in \mathscr{F}$, then we have $\Gamma \in \mathscr{F}$.
\item No amenable group belongs to $\mathscr{F}$.
\end{enumerate}
\end{prop}

\begin{proof}
Let $\Gamma$ and $\Lambda$ be discrete groups with $\Lambda <\Gamma$ and $[\Gamma :\Lambda]<\infty$.
Assuming that $\Gamma$ belongs to $\mathscr{F}$, we prove that so does $\Lambda$.
Pick a measure-preserving action $\Lambda \c (Y, \nu)$ on a standard probability space and a Borel cocycle $\tau \colon \Lambda \times Y\rightarrow \aut^*(T)$ satisfying the three conditions in Definition \ref{defn-c-f}. 
Let $\Gamma \c (X, \mu)$ be the action of $\Gamma$ induced from the action $\Lambda \c (Y, \nu)$. 
Choose representatives $\gamma_1, \gamma_2, \ldots, \gamma_{I}$ of left cosets of $\Lambda$ in $\Gamma$ with $\gamma_1=e$ and $I=[\Gamma :\Lambda]$.
Note that we have $X=\bigsqcup_{i=1}^{I}\gamma_iY$. 
Let us define a Borel cocycle $\rho \colon \Gamma \times X\rightarrow \aut^*(T)$ in the following way: Given $\gamma \in \Gamma$, $y\in Y$ and $i\in \{ 1,\ldots, I\}$, we can uniquely find elements $j\in \{ 1,\ldots, I\}$ and $\lambda \in \Lambda$ with $\gamma \gamma_i=\gamma_j\lambda$ and define $\rho(\gamma, \gamma_iy)=\tau(\lambda, y)$. 
In particular, we have $\rho(\gamma_i, y)=e$ for any $i\in \{ 1,\ldots, I\}$ and $y\in Y$.

We now prove that $\rho$ is in fact a Borel cocycle.
Given any $\gamma, \delta \in \Gamma$ and $x\in X$, we uniquely choose elements $i, j, k\in \{ 1,\ldots, I\}$, $y\in Y$ and $\lambda, \xi \in \Lambda$ with
\[x=\gamma_iy,\quad \delta \gamma_i=\gamma_j\lambda,\quad \gamma \gamma_j=\gamma_k\xi.\]
We thus have
\[\delta x=\delta \gamma_iy=\gamma_j\lambda y,\quad \gamma \delta \gamma_i=\gamma \gamma_j\lambda =\gamma_k\xi \lambda.\]
By definition, the equalities
\[\rho(\gamma, \delta x)=\tau(\xi, \lambda y),\quad \rho(\delta, x)=\tau(\lambda, y),\quad \rho(\gamma \delta, x)=\tau(\xi \lambda, y)\]
holds. 
It follows that $\rho$ is a Borel cocycle. 

The cocycle $\rho$ satisfies the three conditions in Definition \ref{defn-c-f} for the action $\Gamma \c (X, \mu)$. 
There exists a $(\Gamma \ltimes X, \rho)$-invariant Borel map $\varphi \colon X\rightarrow V(T)$ because $\Gamma$ belongs to $\mathscr{F}$. 
The restriction of $\varphi$ to $Y$ is then $(\Lambda \ltimes Y, \tau)$-invariant. 
It follows that $\Lambda$ belongs to $\mathscr{F}$.

The converse of assertion (i) is a direct consequence of the fact shown in the previous paragraph and assertion (ii). 
We now prove assertion (ii). 
Let $\Gamma$ and $\Lambda$ be discrete groups with $\Lambda \lhd \Gamma$ and $\Lambda \in \mathscr{F}$.
Pick a measure-preserving action $\Gamma \c (X, \mu)$ on a standard probability space and a Borel cocycle $\rho \colon \Gamma \times X\rightarrow \aut^*(T)$ as in Definition \ref{defn-c-f}. 
Since $\Lambda$ belongs to $\mathscr{F}$, there exists a $(\Lambda \ltimes X)$-invariant Borel map from $X$ into $V(T)$. 
Let $\varphi_0\colon X\rightarrow S(T)$ be the $(\Lambda \ltimes X)$-invariant Borel map given in Lemma \ref{lem-inv-s}. 
It follows from $\Lambda \lhd \Gamma$ that $\varphi_0$ is also $(\Gamma \ltimes X)$-invariant. 
We can then construct a $(\Gamma \ltimes X)$-invariant Borel map from $X$ into $V(T)$ by using Lemma \ref{lem-inv-ext} and condition (III) in Definition \ref{defn-c-f}. 
It thus turns out that $\Gamma$ belongs to $\mathscr{F}$.

Finally, we prove assertion (iii). 
Let $\Gamma$ be an infinite amenable group and pick an ergodic f.f.m.p.\ action $\Gamma \c (X, \mu)$. 
By a theorem due to Ornstein-Weiss \cite{ow} or Connes-Feldman-Weiss \cite{cfw}, there exists an ergodic f.f.m.p.\ action $\mathbb{Z}\c (X, \mu)$ such that the two groupoids $\Gamma \ltimes X$ and $\mathbb{Z}\ltimes X$ are isomorphic with respect to the identity on $X$. 
Let $T$ be the Cayley graph of $\mathbb{Z}$ with respect to the set of generators $\{ \pm 1\}$, and define a cocycle $\rho \colon \mathbb{Z}\times X\rightarrow \aut^*(T)$ so that for any $n\in \mathbb{Z}$ and $x\in X$, $\rho(n, x)$ is the translation on $T$ defined by $m\mapsto m+n$. 
The cocycle obtained by composing $\rho$ with the isomorphism $\Gamma \ltimes X \simeq \mathbb{Z}\ltimes X$ satisfies the three conditions in Definition \ref{defn-c-f}, while there exists no $(\Gamma \ltimes X)$-invariant Borel map from $X$ into $V(T)$. 
We then see that $\Gamma$ does not belong to $\mathscr{F}$.
\end{proof}


\subsection{Natural maps associated to trees}\label{subsec-nat-maps}

Let $T$ be a connected simplicial tree with at most countably many simplices. 
In this subsection, we collect basic properties of invariant Borel maps valued in the space of probability measures on $\Delta T$. 
We first review some natural maps associated to geometry of $T$. 
For a compact Hausdorff space $Z$, we denote by $M(Z)$ the space of probability measures on $Z$. 
For each Borel subset $W$ of $Z$, we set
\[M(W)=\{ \, \mu\in M(Z)\mid \mu(W)=1\,\},\]
which is a Borel subset of $M(Z)$.

\medskip

\noindent {\bf Map $C\colon V_f(T)\rightarrow S(T)$.} Let $V_f(T)$ be the set of all non-empty finite subsets of $V(T)$, on which $\aut^*(T)$ naturally acts. 
The map $C\colon V_f(T)\rightarrow S(T)$ is defined to be the $\aut^*(T)$-equivariant map associating the barycenter to each element of $V_f(T)$ in the following manner: For each $F\in V_f(T)$, let $\bar{F}\in V_f(T)$ be the set of vertices of the minimal connected subtree of $T$ containing $F$, and put $X_0=\bar{F}$. 
Let us define a subset $X_{i+1}$ of $V(T)$ as $X_i\setminus \partial X_i$ inductively, where $\partial X_i$ denotes the set of all vertices in $X_i$ whose neighbors in $X_i$ consist of exactly one vertex. 
Since $X_0$ is finite, there exists a unique number $i$ such that $X_i$ consists of either a single vertex or two adjacent vertices. 
This $X_i$ is denoted by $C(F)\in S(T)$.
The map $C$ is clearly $\aut^*(T)$-equivariant.     
\medskip

\noindent {\bf Map $\delta C\colon \delta T\rightarrow \delta V(T)$.} Suppose that $\partial T$ contains at least three points. 
We set
\[\delta T=\{\, (x, y, z)\in (\partial T)^{3}\mid x\neq y\neq z\neq x\, \},\]
on which $\aut^*(T)$ acts so that $f(x, y, z)=(f(x), f(y), f(z))$ for any $f\in \aut^*(T)$ and $(x, y, z)\in \delta T$. 
Let $\delta V(T)$ denote the set of all finite subsets of $V(T)$ whose cardinalities are at least three, on which $\aut^*(T)$ naturally acts. 
We define a Borel map $\delta C\colon \delta T\rightarrow \delta V(T)$ so that for each $(x, y, z)\in \delta T$, the set $\delta C(x, y, z)$ consists of all vertices in the union $l_x^y\cup l_y^z\cup l_z^x$ whose distances from the single point in the intersection $l_x^y\cap l_y^z\cap l_z^x$ are at most one.

\medskip

\noindent {\bf Map $M\colon M(\delta V(T))\rightarrow \delta V(T)$.} This is defined to be an $\aut^*(T)$-equivariant Borel map associating to each $\ell^1$-function $f\in M(\delta V(T))$ the union of the points of $\delta V(T)$ on which $f$ takes its maximal value.
\medskip

\noindent {\bf Map $G\colon \partial_2T\times S(T)\rightarrow \delta V(T)$.} This is an $\aut^*(T)$-equivariant Borel map defined so that for any $a\in \partial_2T$ and $s\in S(T)$, the set $G(a, s)$ consists of all vertices in the geodesic(s) connecting a point of $a$ with $s$ whose distances from $s$ are at least one and at most three.
\medskip

Suppose that we have a measure-preserving action $\Gamma \c (X, \mu)$ of a discrete group on a standard probability space and a Borel cocycle $\rho \colon \Gamma \times X\rightarrow \aut^*(T)$ satisfying condition (I) in Definition \ref{defn-c-f}. 
Given a subgroupoid $\mathcal{S}$ of $\Gamma \ltimes X$ of infinite type, we study $\cal{S}$-invariant Borel maps into $M(\Delta T)$ in the following lemmas.
For each measure $\mu \in M(\Delta T)$, we denote by ${\rm supp}(\mu)$ the support of $\mu$.

\begin{lem}\label{lem-two-supp}
Let $\Gamma \c (X, \mu)$ be a measure-preserving action of a discrete group on a standard probability space and $\mathcal{G}$ the associated groupoid. 
Suppose that we have
\begin{itemize}
\item a tree $T$ with $\partial T\neq \emptyset$ and a Borel cocycle $\rho \colon \Gamma \times X\rightarrow \aut^*(T)$ satisfying condition (I) in Definition \ref{defn-c-f};
\item a Borel subset $A$ of $X$ with positive measure and a subgroupoid $\mathcal{S}$ of $(\mathcal{G})_A$ of infinite type; and
\item an $\cal{S}$-invariant Borel map $\varphi \colon A\rightarrow M(\Delta T)$.
\end{itemize}
Then the following assertions hold:
\begin{enumerate}
\item If $\varphi(x)(\partial T)=1$ for a.e.\ $x\in A$, then ${\rm supp}(\varphi(x))$ consists of at most two points of $\partial T$ for a.e.\ $x\in A$. 
The map $\varphi$ thus induces an $\cal{S}$-invariant Borel map from $A$ into $\partial_2T$.
\item There exists an essentially unique Borel partition $A=A_1\sqcup A_2$ such that $\varphi(x)(\partial T)=1$ for a.e.\ $x\in A_1$ and $\varphi(y)(V(T))=1$ for a.e.\ $y\in A_2$.
\end{enumerate}
\end{lem}

\begin{proof}
To prove assertion (i), we may assume that $\partial T$ contains at least three points. 
If there were a Borel subset $B\subset A$ of positive measure with $\left|{\rm supp}(\varphi(x))\right|\geq 3$ for each $x\in B$, then the restriction of the product measure $\varphi(x)^3$ on $(\partial T)^{3}$ to $\delta T$ would not be zero for each $x\in B$. 
By normalizing the restricted measure and by composing the map
\[M\circ (\delta C)_*\colon M(\delta T)\rightarrow \delta V(T),\]
we obtain an $\cal{S}$-invariant Borel map from $B$ into $\delta V(T)$. 
This contradicts condition (I) for $\rho$ in Definition \ref{defn-c-f} because $\cal{S}$ is of infinite type.
Assertion (i) is proved.

We next prove assertion (ii). 
If there were a Borel subset $B\subset A$ such that for any $x\in B$, both $\varphi(x)(\partial T)$ and $\varphi(x)(V(T))$ are positive, then we would have two $\cal{S}$-invariant Borel maps
\[\psi_1\colon B\rightarrow \partial_2T,\quad \psi_2\colon B\rightarrow S(T).\] 
The latter is obtained by using the map $C\colon V_f(T)\rightarrow S(T)$. 
The map $B\ni x\mapsto G(\psi_1(x), \psi_2(x))\in \delta V(T)$ is then $\cal{S}$-invariant. 
We can deduce a contradiction as in the proof of assertion (i).
Assertion (ii) thus follows.
\end{proof}

\begin{lem}\label{lem-inv-spe}
We define $\cal{G}$, $T$, $\rho$, $A$, $\cal{S}$ and $\varphi$ as in Lemma \ref{lem-two-supp}.
Let $A=A_1\sqcup A_2$ be the partition in Lemma \ref{lem-two-supp} (ii).
Then we have two $\cal{S}$-invariant Borel maps
\[\varphi_1\colon A_1\rightarrow \partial_2T,\quad \varphi_2\colon A_2\rightarrow S(T)\]
satisfying the following two conditions:
\begin{enumerate}
\item[(a)] For each $i=1, 2$, let $B_i$ be a Borel subset of $A_i$ with positive measure and $\psi_i\colon B_i\rightarrow M(\Delta T)$ an $\cal{S}$-invariant Borel map.
Then the inclusions
\[{\rm supp}(\psi_1(x))\subset \varphi_1(x),\quad {\rm supp}(\psi_2(y))\subset \varphi_2(y)\]
hold for a.e.\ $x\in B_1$ and a.e.\ $y\in B_2$.
\item[(b)] If $\cal{T}$ is a subgroupoid of $(\cal{G})_A$ with $\cal{S}\lhd \cal{T}$, then $\varphi_1$ and $\varphi_2$ are $\cal{T}$-invariant.
\end{enumerate}
\end{lem}

\begin{proof}
The construction of $\varphi_1$ is almost the same as that of $\varphi_0$ in Lemma \ref{lem-inv-s}. 
Let $I$ denote the set of all $\cal{S}$-invariant Borel maps from $A_1$ into $M(\Delta T)$. 
It follows from Lemma \ref{lem-two-supp} (ii) that each $\varphi \in I$ satisfies the equality $\varphi(x)(\partial T)=1$ for a.e.\ $x\in A_1$. 
For each $\varphi \in I$, we set
\[S_{\varphi}=\{\, x\in A_1\mid \left|{\rm supp}(\varphi(x))\right|=2\,\}.\]
There exists $\varphi_1\in I$ with $\mu(S_{\varphi_1})=\sup_{\varphi \in I}\mu(S_{\varphi})$.
This map $\varphi_1$ then satisfies condition (a). 
The existence of $\varphi_2$ follows from Lemma \ref{lem-inv-s}.

If $\cal{T}$ is a subgroupoid of $(\cal{G})_A$ with $\cal{S}\lhd \cal{T}$, then $\varphi_2$ is $\cal{T}$-invariant by Lemma \ref{lem-inv-s} (ii). 
An argument of the same kind as in the proof of Lemma \ref{lem-inv-s} (ii) shows that $\varphi_1$ is $\cal{T}$-invariant.
\end{proof}


\subsection{Examples of groups in Class $\mathscr{F}$}\label{subsec-c-f-ex}

Combining results on invariant Borel maps proved in Section \ref{subsec-nat-maps}, we obtain the following hereditary property of groups in Class $\mathscr{F}$.

\begin{prop}\label{prop-ex-c}
Let $\Gamma$ be a discrete group.
Suppose that we have a positive integer $n$ and infinite subgroups $\Gamma_i$ and $N_i$ for $i=1,\ldots, n$ satisfying the following three conditions:
\begin{itemize}
\item For each $i=1,\ldots, n$, we have $N_i\lhd \Gamma_i$.
\item For each $i=1,\ldots, n-1$, we have $N_{i+1}<\Gamma_i$ and $N_1<\Gamma_n$.
\item $\Gamma$ is generated by $\Gamma_1,\ldots, \Gamma_n$. 
\end{itemize}
If $\Gamma_1$ belongs to $\mathscr{F}$, then so does $\Gamma$.
\end{prop}

\begin{proof}
Let $\Gamma \c (X, \mu)$ be a measure-preserving action on a standard probability space, and let $\rho \colon \Gamma \times X\rightarrow \aut^*(T)$ be a Borel cocycle satisfying the three conditions in Definition \ref{defn-c-f}. 
For each $i=1,\ldots, n$, we set
\[\cal{G}_i=\Gamma_i\ltimes X,\quad \cal{N}_i=N_i\ltimes X.\]
Since $\Gamma_1$ belongs to $\mathscr{F}$, there exists a $\cal{G}_1$-invariant Borel map from $X$ into $V(T)$, which is obviously $\cal{N}_2$-invariant. 
Applying Lemma \ref{lem-inv-s} (ii), we obtain a $\cal{G}_2$-invariant Borel map from $X$ into $S(T)$. 
Repeating this process, for each $i=1,\ldots, n$, one can find a $\cal{G}_i$-invariant Borel map from $X$ into $S(T)$. 
By Lemma \ref{lem-inv-s}, for each $i=1,\ldots, n$, there exist a $\cal{G}_i$-invariant Borel map $\varphi_i\colon X\rightarrow S(T)$ and an $\cal{N}_i$-invariant $\psi_i\colon X\rightarrow S(T)$ satisfying the following three conditions:
\begin{itemize}
\item If $\varphi \colon X\rightarrow S(T)$ is a $\cal{G}_i$-invariant Borel map, then we have $\varphi(x)\subset \varphi_i(x)$ for a.e.\ $x\in X$.
\item If $\psi \colon X\rightarrow S(T)$ is an $\cal{N}_i$-invariant Borel map, then we have $\psi(x)\subset \psi_i(x)$ for a.e.\ $x\in X$.
\item $\psi_i$ is $\cal{G}_i$-invariant. 
\end{itemize}
The first and third conditions imply that for each $i$, the inclusion $\psi_i(x)\subset \varphi_i(x)$ holds for a.e.\ $x\in X$. 
Since $\varphi_i$ is also $\cal{N}_i$-invariant, the second condition implies that for each $i$, the inclusion $\varphi_i(x)\subset \psi_i(x)$ holds for a.e.\ $x\in X$. 
It follows that $\varphi_i$ and $\psi_i$ coincide a.e.\ on $X$ for each $i$. 
On the other hand, for each $i=1,\ldots, n-1$, the inclusion $\cal{N}_{i+1}<\cal{G}_i$ implies that $\varphi_i(x)\subset \psi_{i+1}(x)$ for a.e.\ $x\in X$. 
Similarly, we have $\varphi_n(x)\subset \psi_1(x)$ for a.e.\ $x\in X$. 
The inclusion
\[\varphi_1(x)\subset \varphi_2(x)\subset \cdots \subset \varphi_n(x)\subset \varphi_1(x)\]   
thus holds for a.e.\ $x\in X$, and all $\varphi_i$ coincide a.e.\ on $X$. 
This single map is then $(\Gamma \ltimes X)$-invariant since $\Gamma \ltimes X$ is generated by $\cal{G}_1,\ldots, \cal{G}_n$.
\end{proof}

We now present examples of groups in Class $\mathscr{F}$.

\begin{prop}\label{prop-exa-c}
A discrete group $\Gamma$ belongs to Class $\mathscr{F}$ if it is non-amenable and satisfies one of the following two conditions:
\begin{enumerate}
\item[(a)] There is an infinite, amenable and normal subgroup $N$ of $\Gamma$.
\item[(b)] There is an infinite normal subgroup $N$ of $\Gamma$ such that the pair $(\Gamma, N)$ satisfies relative property (T) of Kazhdan-Margulis.
\end{enumerate}
\end{prop}

\begin{proof}
Let $\Gamma \c (X, \mu)$ be a measure-preserving action on a standard probability space, and let $\rho \colon \Gamma \times X\rightarrow \aut^*(T)$ be a Borel cocycle satisfying the three conditions in Definition \ref{defn-c-f}. 
We set
\[\cal{G}=\Gamma \ltimes X,\quad \cal{N}=N\ltimes X.\]

In case (a), since $\cal{N}$ is amenable, there exists an $\cal{N}$-invariant Borel map from $X$ into $M(\Delta T)$. 
By Lemma \ref{lem-two-supp} (ii), there exists a Borel partition $X=A_1\sqcup A_2$ such that one can construct an $\cal{N}$-invariant Borel map from $A_1$ into $\partial_2T$ and an $\cal{N}$-invariant Borel map from $A_2$ into $S(T)$. 
Since $\cal{G}$ is nowhere amenable, the equality $X=A_2$ holds up to null sets by Lemma \ref{lem-inv-spe} and condition (II) for $\rho$. 
We therefore obtain a $\cal{G}$-invariant Borel map from $X$ into $S(T)$. 
By using Lemma \ref{lem-inv-ext} and condition (III) in Definition \ref{defn-c-f}, one obtains a $\cal{G}$-invariant Borel map from $X$ into $V(T)$. 
It follows that $\Gamma$ belongs to $\mathscr{F}$.

In case (b), one can find an $\cal{N}$-invariant Borel map from $X$ into $V(T)$ by a relative version of a theorem due to Adams-Spatzier \cite{adams-spa}. 
Lemma \ref{lem-inv-spe} then implies existence of a $\cal{G}$-invariant Borel map from $X$ into $V(T)$.
\end{proof}

\begin{prop}\label{prop-c-f-mcg}
Let $S$ be a surface of genus $g$ with $p$ boundary components. 
If $3g+p\geq 5$, then any finite index subgroup of $\mod^*(S)$ belongs to Class $\mathscr{F}$. 
\end{prop}

To prove this proposition, let us introduce terminology. 
We say that a simple closed curve in $S$ is {\it essential} if it is neither homotopic to a point of $S$ nor isotopic to a boundary component of $S$. 
Let $V(S)$ denote the set of isotopy classes of essential simple closed curves in $S$.
Let
\[I\colon V(S)\times V(S)\rightarrow \mathbb{Z}_{\geq 0}\]
denote the geometric intersection number, i.e., the minimal cardinality of the intersection of representatives for two elements of $V(S)$.
\begin{figure}
\begin{center}
\includegraphics[width=12cm]{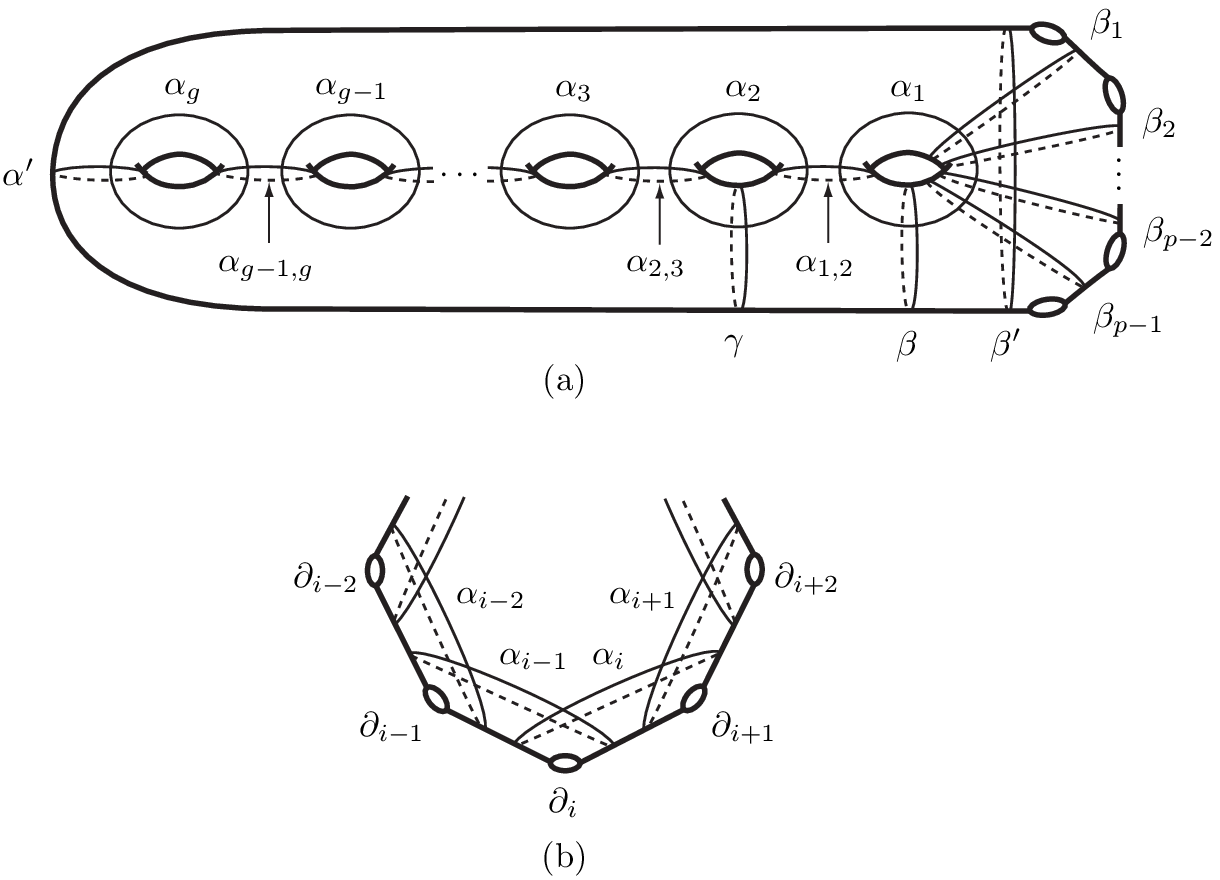}
\caption{}\label{fig-seq}
\end{center}
\end{figure}

The {\it mapping class group} $\mod(S)$ of $S$ is defined to be the group of isotopy classes of orientation-preserving homeomorphisms from $S$ onto itself. 
The {\it pure mapping class group} $\pmod(S)$ of $S$ is defined to be the group of isotopy classes of orientation-preserving homeomorphisms from $S$ onto itself fixing each boundary component of $S$ as a set. 
Note that $\mod(S)$ and $\pmod(S)$ are finite index subgroups of $\mod^*(S)$.
If $g\geq 1$, then $\pmod(S)$ is generated by the Dehn twists about curves in the family
\[\{ \alpha_i\}_{i=1}^g\cup \{ \alpha_{j, j+1}\}_{j=1}^{g-1}\cup \{ \beta_k\}_{k=1}^{p-1}\cup \{ \beta, \gamma \}\]
described in Figure \ref{fig-seq} (a), where $\gamma$ is excluded if $g=1$. 
If $g=0$, then $\mod(S)$ is generated by the half twists about curves in the family $\{ \alpha_i\}_{i=1}^p$ described in Figure \ref{fig-seq} (b). 
We refer to \cite{gervais} and Theorem 4.5 in \cite{birman} for these facts, respectively.

\begin{proof}[Proof of Proposition \ref{prop-c-f-mcg}]
We first assume $g\geq 1$.
We claim that it suffices to find a positive integer $n$ and $\delta_1,\ldots, \delta_n\in V(S)$ satisfying the following two conditions:
\begin{enumerate}
\item[(1)] $I(\delta_i, \delta_{i+1})=0$ for each $i=1,\ldots, n-1$ and $I(\delta_n, \delta_1)=0$.
\item[(2)] $\pmod(S)$ is generated by $t_1,\ldots, t_n$, where $t_i\in \pmod(S)$ is the Dehn twist about $\delta_i$ for each $i$.
\end{enumerate}
For each $i=1,\ldots, n$, let $N_i$ be the infinite cyclic subgroup generated by $t_i$, and let $\Gamma_i$ be the stabilizer of $\delta_i$ in $\pmod(S)$.
It follows from the equality
\[ft_{\delta}f^{-1}=t_{f\delta}, \quad \forall f\in \mod(S),\ \forall \delta \in V(S)\]
(see Lemma 4.1.C in \cite{iva-mcg}) that for each $i=1,\ldots, n$, $N_i$ is contained in the center of $\Gamma_i$.
Condition (1) implies $N_{i+1}<\Gamma_i$ for each $i=1,\ldots, n-1$ and $N_1<\Gamma_n$.
It follows from $3g+p\geq 5$ that each $\Gamma_i$ is non-amenable.
Propositions \ref{prop-ex-c} and \ref{prop-exa-c} show that $\pmod(S)$ belongs to $\mathscr{F}$, and therefore so does any finite index subgroup of $\mod^*(S)$ by Proposition \ref{prop-basic-c} (i).
The claim follows.

We now find a sequence of curves satisfying conditions (1) and (2).
If $g\geq 2$, then the sequence
\[\alpha_1,\ \alpha_2, \ldots,\ \alpha_g,\ \beta_1,\ \beta_2,\ldots,\ \beta_{p-1},\ \beta,\ \alpha_{1, 2},\ \alpha_{2, 3},\ldots,\ \alpha_{g-1, g},\ \gamma\]
is a desired one.
Assume $g=1$.
We have $p\geq 2$ and define an essential curve $\beta'$ as in Figure \ref{fig-seq} (a).
The sequence $\alpha_1$, $\beta'$, $\beta$, $\beta_1$, $\beta_2,\ldots,$ $\beta_{p-1}$, $\beta$, $\beta'$ is a desired one.

We assume $g=0$.
We then have $p\geq 5$.
As in the first paragraph in the proof, it is shown that it suffices to find a positive integer $n$ and $\delta_1,\ldots, \delta_n\in V(S)$ satisfying condition (1) and the following condition:
\begin{enumerate}
\item[(3)] $\mod(S)$ is generated by $h_1,\ldots, h_n$, where $h_i\in \mod(S)$ is the half twist about $\delta_i$ for each $i$.
\end{enumerate}
If $p$ is odd, then the sequence $\alpha_1$, $\alpha_3,\ldots,$ $\alpha_p$, $\alpha_2$, $\alpha_4,\ldots,$ $\alpha_{p-1}$ is a desired one.
If $p$ is even, then the sequence $\alpha_1$, $\alpha_3,\ldots,$ $\alpha_{p-1}$, $\alpha_2$, $\alpha_4,\ldots,$ $\alpha_p$, $\alpha_3$ is a desired one.
\end{proof}


\section{Class $\mathscr{A}$}\label{sec-c-a}

In Section \ref{sec-coup-rigid}, we will prove coupling rigidity of an amalgamated free product $\Gamma_1\ast_A\Gamma_2$ of discrete groups such that $\Gamma_1$ and $\Gamma_2$ belong to Class $\mathscr{F}$ and $A$ is an infinite, almost malnormal subgroup of both $\Gamma_1$ and $\Gamma_2$. 
It is however often difficult to find almost malnormal subgroups of a given group. 
We now introduce Class $\mathscr{A}$ of discrete groups, smaller than Class $\mathscr{F}$, so that coupling rigidity of $\Gamma_1\ast_A\Gamma_2$ will also be proved when $\Gamma_1$ and $\Gamma_2$ belong to $\mathscr{A}$ and a milder condition is imposed on the subgroup $A$. 
In this section, we discuss basic properties and examples of groups in Class $\mathscr{A}$.

\begin{defn}\label{defn-c-a}
We say that a discrete group $\Gamma$ belongs to Class $\mathscr{A}$ if $\Gamma$ is infinite and the following statement holds: Let $\Gamma \c (X, \mu)$ be a measure-preserving action on a standard probability space and $\mathcal{G}$ the associated groupoid. 
Suppose that we have a tree $T$ with $\partial T\neq \emptyset$ and a Borel cocycle $\rho \colon \Gamma \times X\rightarrow \aut^*(T)$ satisfying conditions (II) and (III) in Definition \ref{defn-c-f} and the following condition:
\begin{enumerate}
\item[(Ia)] For any three mutually distinct vertices $v_1, v_2, v_3\in V(T)$, the subgroupoid of $\cal{G}$ defined as
\[\{\, (\gamma, x)\in \mathcal{G}\mid \rho(\gamma, x)v_i=v_i \ \textrm{for\ any}\ i=1, 2, 3\, \}\]
is amenable.
\end{enumerate}
Then there exists a $(\cal{G}, \rho)$-invariant Borel map $\varphi \colon X\rightarrow V(T)$.
\end{defn}

\subsection{Fixed point maps}

Let $\mathcal{G}$ be a discrete measured groupoid on a standard probability space $(X, \mu)$. 
Suppose that we have a tree $T$ with $\partial T\neq \emptyset$ and a Borel cocycle $\rho \colon \cal{G}\rightarrow \aut^*(T)$. In this subsection, we introduce a canonical $\cal{G}$-invariant Borel map, called the fixed point map for $\cal{G}$. 
This map will be used to find examples of groups in Class $\mathscr{A}$.

Let $\cal{C}T$ denote the set of all non-empty connected subtrees in $T$, which can be identified with a subset of $P(V(T))$, the power set of $V(T)$.
We define a map
\[\xi \colon \cal{C}T\times \cal{C}T\rightarrow \cal{C}T\]
as follows: Pick $T_1, T_2\in \cal{C}T$.
If $T_1\cap T_2\neq \emptyset$, set $\xi(T_1, T_2)=T_1\cap T_2$. 
If $T_1\cap T_2=\emptyset$, define $\xi(T_1, T_2)$ as the shortest path in $T$ joining $T_1$ and $T_2$.

\begin{lem}
In the above notation,
\begin{enumerate}
\item the set $\cal{C}T$ is a Borel subset of $P(V(T))$, where $P(V(T))$ is equipped with the product topology.
\item the map $\xi \colon \cal{C}T\times \cal{C}T\rightarrow \cal{C}T$ is Borel.
\end{enumerate}
\end{lem}

\begin{proof}
Let $\cal{L}$ be the set of geodesics in $T$ whose lengths are at least one and finite. 
For each $l\in \cal{L}$, we identify $l$ with the set of vertices of $T$ through which $l$ passes and denote by $\partial l$ the set of end vertices of $l$. 
We then have the equality
\[\cal{C}T\cup \{ \emptyset \}=\bigcap_{l\in \cal{L}}(\{\, R\in P(V(T))\mid \left|R\cap \partial l\right|<2\,\}\cup \{\, S\in P(V(T))\mid l\subset S\,\}),\]
which shows assertion (i).
It then follows that the set $\cal{A}$ consisting of all elements $(T_1, T_2)$ of $\cal{C}T\times \cal{C}T$ with $T_1\cap T_2\neq \emptyset$ is Borel and that the map $\xi$ is Borel on $\cal{A}$.

Pick $l\in \cal{L}$ and put $\partial l=\{ v_1, v_2\}$. 
For each $i=1, 2$, let $S_i$ be the maximal connected subtree of $T$ containing $v_i$ and no other vertices of $l$. 
The set $\xi^{-1}(\{ l\} )\setminus \cal{A}$ then consists of all elements $(T_1, T_2)$ of $\cal{C}T\times \cal{C}T$ with either $v_1\in T_1\subset S_1$ and $v_2\in T_2\subset S_2$ or $v_2\in T_1\subset S_2$ and $v_1\in T_2\subset S_1$. 
Assertion (ii) thus follows.
\end{proof}

We say that a pair $(v, A)$ of $v\in V(T)$ and a Borel subset $A$ of $X$ with positive measure is {\it $\cal{G}$-invariant} if there exists a countable Borel partition $A=\bigsqcup A_n$ such that for each $n$, the constant map on $A_n$ with its value $v$ is $\cal{G}$-invariant. 
For each $v\in V(T)$, there exists an essentially unique Borel subset $X_v$ of $X$ such that if $(v, A)$ is a $\cal{G}$-invariant pair, then $A$ is essentially contained in $X_v$.

The fixed point map for $\cal{G}$ is defined by collecting all $\cal{G}$-invariant pairs as follows. 
We define a Borel map
\[\theta=\theta(\cal{G}, \rho)\colon X\rightarrow P(V(T))\]
so that we have $v\in \theta(x)$ if $x\in X_v$. Let us call $\theta$ the {\it fixed point map} for $\cal{G}$.

\begin{lem}\label{lem-fix-point}
In the above notation, the following assertions hold:
\begin{enumerate}
\item If $A$ is a Borel subset of $X$ with positive measure, then the fixed point map for $(\cal{G})_A$ is equal to the restriction of $\theta$ to $A$.
\item If $\cal{S}$ is a normal subgroupoid of $\cal{G}$, then the fixed point map for $\cal{S}$ is $\cal{G}$-invariant.
\item We have $\theta(x)\in \cal{C}T\cup \{ \emptyset \}$ for a.e.\ $x\in X$. 
\end{enumerate}
\end{lem}

\begin{proof}
For each $v\in V(T)$ and each Borel subset $B$ of $A$ with positive measure, the pair $(v, B)$ is $(\cal{G})_A$-invariant if and only if it is $\cal{G}$-invariant. 
Assertion (i) thus follows. 
Assertion (ii) can be verified along an argument of the same kind as in the proof of Lemma \ref{lem-inv-s} (ii). 
Assertion (iii) holds because for any two vertices $u$, $v$ of $T$, any automorphism of $T$ fixing each of $u$ and $v$ fixes each vertex in the geodesic connecting $u$ and $v$.
\end{proof}


\subsection{Examples of groups in Class $\mathscr{A}$}

We first present hereditary properties of groups in Class $\mathscr{A}$.

\begin{prop}\label{prop-per-c-a}
Class $\mathscr{A}$ satisfies the following properties:
\begin{enumerate}
\item We have the inclusion $\mathscr{A}\subset \mathscr{F}$.
\item Let $\Gamma$ and $\Lambda$ be discrete groups with $\Lambda <\Gamma$ and $[\Gamma :\Lambda]<\infty$. 
Then $\Gamma \in \mathscr{A}$ if and only if $\Lambda \in \mathscr{A}$.
\item If $\Gamma$ and $\Lambda$ are discrete groups with $\Lambda \lhd \Gamma$ and $\Lambda \in \mathscr{A}$, then $\Gamma \in \mathscr{A}$.
\end{enumerate}
\end{prop}

\begin{proof}
Assertion (i) follows by definition. 
We prove assertion (iii). 
Let $\Gamma$ and $\Lambda$ be discrete groups with $\Lambda \lhd \Gamma$ and $\Lambda \in \mathscr{A}$. 
To prove $\Gamma \in \mathscr{A}$, pick an action $\Gamma \c (X, \mu)$ and a cocycle $\rho \colon \cal{G}\rightarrow \aut^*(T)$ as in Definition \ref{defn-c-a}, where we put $\cal{G}=\Gamma \ltimes X$.
It follows from $\Lambda \in \mathscr{A}$ that $\Lambda$ is non-amenable and there exists a $(\Lambda \ltimes X)$-invariant Borel map from $X$ into $V(T)$. 
Using condition (Ia) for $\rho$ and the method to construct the map $\varphi_0$ in Lemma \ref{lem-inv-s}, we can find a $(\Lambda \ltimes X)$-invariant Borel map $\varphi_0\colon X\rightarrow S(T)$ such that if $\varphi \colon X\rightarrow V(T)$ is a $(\Lambda \ltimes X)$-invariant Borel map, then the inclusion $\varphi(x)\subset \varphi_0(x)$ holds for a.e.\ $x\in X$. 
This $\varphi_0$ can be shown to be $(\Gamma \ltimes X)$-invariant since we have $\Lambda \lhd \Gamma$. 
Assertion (iii) then follows. 

Using assertion (iii), we can show assertion (ii) along the same argument as in the proof of Proposition \ref{prop-basic-c} (i). 
\end{proof}

The following proposition gives a sufficient condition for a discrete group to be in $\mathscr{A}$. 
For a group $G$ and subgroups $H_1$ and $H_2$ of $G$, let us write
\[H_1\bowtie H_2\]
if we have $H_1\lhd H$ and $H_2\lhd H$, where $H$ is the subgroup of $G$ generated by $H_1$ and $H_2$.

\begin{prop}\label{prop-ex-c-a}
A discrete group $\Gamma$ belongs to Class $\mathscr{A}$ if there are subgroups $N_1, N_2,\ldots, N_k$ of $\Gamma$ satisfying $k\geq 2$ and the following five conditions:
\begin{itemize}
\item For each $i=1,\ldots, k$, $N_i$ is non-amenable.
\item For each $i=1,\ldots, k-1$, we have $N_i\bowtie N_{i+1}$.
\item $N_1$ contains infinite amenable subgroups $N_1^1$ and $N_1^2$ which generate $N_1$ and satisfy $N_1^j\bowtie N_2$ for each $j=1, 2$.
\item For each $j=1, 2$, $N_1^j$ contains an infinite subgroup $M^j$ such that $M^j\bowtie N_2$, and $M^1$ and $M^2$ generate the free product of them.
\item $\Gamma$ is generated by $N_1, N_2,\ldots, N_k$.
\end{itemize}
\end{prop}

To prove this proposition, we need the following:

\begin{lem}\label{lem-m-non-ame}
Let $M_1$ and $M_2$ be infinite amenable groups and define $M$ to be the free product $M_1\ast M_2$ of them. 
Suppose that we have a measure-preserving action $M\c (X, \mu)$ on a standard probability space, and let $\cal{M}$ denote the associated groupoid. 
For each $i=1, 2$, let $\cal{M}_i$ denote the subgroupoid of $\cal{M}$ associated to the action of $M_i$. 
Then for any Borel subset $A$ of $X$ with positive measure, the groupoid $(\cal{M}_1)_A\vee (\cal{M}_2)_A$ is non-amenable.  
\end{lem}

\begin{proof}
We define $T$ to be the Bass-Serre tree associated to the decomposition $M=M_1\ast M_2$ (see the paragraph right after Theorem \ref{thm-coup-rigid} for a definition of $T$).
Let $v_1, v_2\in V(T)$ denote the vertices whose stabilizers in $M$ is equal to $M_1$ and $M_2$, respectively. 
For each $i=1, 2$, let ${\rm Lk}(v_i)$ denote the set of vertices in the link of $v_i$ in $T$.
Let 
\[L_i\colon \Delta T\setminus \{ v_i\}\rightarrow {\rm Lk}(v_i)\]
be the $M_i$-equivariant map associating to each point $x$ in $\Delta T\setminus \{ v_i\}$ the single point in the intersection of ${\rm Lk}(v_i)$ with the geodesic between $x$ and $v_i$.
Using this map $L_i$ and the assumption that $M_i$ is infinite, we can show that for each $i=1, 2$, the constant map on $A$ with its value the Dirac measure on $v_i$ is the only $\cal{M}_i$-invariant Borel map from $A$ into $M(\Delta T)$.

We put $\cal{M}_0=(\cal{M}_1)_A\vee (\cal{M}_2)_A$.
If $\cal{M}_0$ were amenable, then there would exist an $\cal{M}_0$-invariant Borel map from $A$ into $M(\Delta T)$, which is $\cal{M}_1$-invariant and $\cal{M}_2$-invariant. 
This is a contradiction.
\end{proof}

\begin{proof}[Proof of Proposition \ref{prop-ex-c-a}]
Let $\Gamma \c (X, \mu)$ be a measure-preserving action on a standard probability space and $\mathcal{G}$ the associated groupoid. 
Let $T$ be a tree with $\partial T\neq \emptyset$ and $\rho \colon \Gamma \times X\rightarrow \aut^*(T)$ a Borel cocycle satisfying condition (Ia) in Definition \ref{defn-c-a} and conditions (II) and (III) in Definition \ref{defn-c-f}. 
We set
\[\cal{N}_i=N_i\ltimes X,\quad \cal{N}_1^j=N_1^j\ltimes X\]
for each $i=1,\ldots, k$ and each $j=1, 2$. 

\begin{claim}\label{claim-inv-n}
For each $j=1, 2$, there exists an $\cal{N}_1^j$-invariant Borel map from $X$ into $V(T)$.
\end{claim}

\begin{proof}
We prove the claim in the case of $j=1$. 
A verbatim argument can be applied in the other case.
Let $I$ be the set of all $\cal{N}_1^1$-invariant Borel maps from $X$ into $M(\Delta T)$. 
Since $N_1^1$ is amenable, the set $I$ is non-empty. For each $\varphi \in I$, we put
\[S_{\varphi}=\{\, x\in X\mid \varphi(x)(V(T))>0\, \}.\]
There exists $\psi \in I$ with $\mu(S_{\psi})=\sup_{\varphi \in I}\mu(S_{\varphi})$. We then have $\psi(x)(\partial T)=1$ for a.e.\ $x\in Y$, where we put $Y=X\setminus S_{\psi}$.

Assuming $\mu(Y)>0$, we deduce a contradiction. 
If we furthermore assume that there exists a Borel subset $A$ of $Y$ with positive measure and $\left|{\rm supp}(\psi(x))\right|\geq 3$ for each $x\in A$, then we can find an $\cal{N}_1^1$-invariant Borel map from $A$ into $V(T)$ by using the maps $\delta C\colon \delta T\rightarrow \delta V(T)$ and $C\colon V_f(T)\rightarrow S(T)$ introduced in Section \ref{subsec-nat-maps}, and this contradicts the maximality of $S_{\psi}$. 
We thus see that ${\rm supp}(\psi(x))$ consists of at most two points for a.e.\ $x\in Y$ and that $\psi$ induces an $\cal{N}_1^1$-invariant Borel map from $Y$ into $\partial_2T$. 
As in the proof of Lemma \ref{lem-inv-spe}, we can then find an $\cal{N}_2$-invariant map from $Y$ into $\partial_2T$. 
Since $N_2$ is non-amenable, this contradicts condition (II) for $\rho$.
\end{proof}

For each $j=1, 2$, let $\theta_1^j\colon X\rightarrow \cal{C}T\cup \{ \emptyset \}$ be the fixed point map for $\cal{N}_1^j$. 
Note that Claim \ref{claim-inv-n} implies $\theta_1^j(x)\in \cal{C}T$ for a.e.\ $x\in X$. 
We set
\[Z=\{\, x\in X\mid \theta_1^1(x)\cap \theta_1^2(x)\neq \emptyset \,\},\quad W=X\setminus Z.\]

\begin{claim}
There exists an $\cal{N}_2$-invariant Borel map from $X$ into $S(T)$.
\end{claim}

\begin{proof}
We can find a countable Borel partition $Z=\bigsqcup_nZ_n$ and a vertex $v_n\in V(T)$ such that for each $n$, the constant map on $Z_n$ with its value $v_n$ is $\cal{N}_1^1$-invariant and $\cal{N}_1^2$-invariant. 
We set
\[\cal{M}^j=M^j\ltimes X,\quad \cal{M}_n=(\cal{M}^1)_{Z_n}\vee (\cal{M}^2)_{Z_n}\]
for each $j=1, 2$ and each $n$. 
Since we have $M^j<N_1^j$ for each $j=1, 2$, for each $n$, there exists an $\cal{M}_n$-invariant Borel map from $Z_n$ into $V(T)$. 
Lemma \ref{lem-m-non-ame} shows that $\cal{M}_n$ is nowhere amenable. 
Along an argument of the same kind as in the proof of Lemma \ref{lem-inv-s} (i), using condition (Ia) for $\rho$, we can construct a Borel map $\varphi_n\colon Z_n\rightarrow S(T)$ which is $\cal{M}_n$-invariant and $\cal{N}_2$-invariant since we have $\cal{M}_n\lhd \cal{M}_n\vee (\cal{N}_2)_{Z_n}$ for each $n$. 
We therefore obtain an $\cal{N}_2$-invariant Borel map $\varphi_0\colon Z\rightarrow S(T)$.

On the other hand, the restrictions of $\theta_1^1$ and $\theta_1^2$ to $W$ are $\cal{N}_2$-invariant by Lemma \ref{lem-fix-point} (ii). 
The map $\varphi_1\colon W\rightarrow \cal{C}T$ defined by $\varphi_1(x)=\xi(\theta_1^1(x), \theta_1^2(x))$ for each $x\in W$ is also $\cal{N}_2$-invariant. 
Note that for a.e.\ $x\in W$, the tree $\varphi_1(x)$ consists of at most finitely many vertices.
Using this $\varphi_1$ and the map $C\colon V_f(T)\rightarrow S(T)$ introduced in Section \ref{subsec-nat-maps}, we obtain an $\cal{N}_2$-invariant Borel map from $W$ into $S(T)$.
\end{proof}

Along an argument of the same kind as in the proof of Lemma \ref{lem-inv-s} (i), using nowhere amenability of $\cal{N}_2$ and condition (Ia) for $\rho$, we can find an $\cal{N}_2$-invariant Borel map $\psi_0\colon X\rightarrow S(T)$ satisfying the following condition: If $A$ is a Borel subset of $X$ with positive measure and if $\psi \colon A\rightarrow V(T)$ is an $\cal{N}_2$-invariant Borel map, then we have the inclusion $\psi(x)\subset \psi_0(x)$ for a.e.\ $x\in A$. 
Let us call this condition for $\psi_0$ the {\it maximal condition} as an $\cal{N}_2$-invariant Borel map.

We now prove that $\psi_0$ is $\cal{G}$-invariant, and this concludes the proposition. 
As in Lemma \ref{lem-inv-s} (ii), the maximal condition implies that $\psi_0$ is also $\cal{N}_1$-invariant and $\cal{N}_3$-invariant because we have $\cal{N}_2\lhd \cal{N}_2\vee \cal{N}_i$ for each $i=1, 3$. 
As in the previous paragraph, we can find an $\cal{N}_3$-invariant Borel map $\psi_3\colon X\rightarrow S(T)$ with the maximal condition as an $\cal{N}_3$-invariant Borel map. 
The map $\psi_3$ is then $\cal{N}_2$-invariant because we have $\cal{N}_3\lhd \cal{N}_2\vee \cal{N}_3$. 
The maximal conditions for $\psi_0$ and $\psi_3$ imply the equality $\psi_0(x)=\psi_3(x)$ for a.e.\ $x\in X$. 
Repeating this argument, we conclude that $\psi_0$ is $\cal{N}_i$-invariant for each $i=1, 2,\ldots, k$ and is $\cal{G}$-invariant because $\Gamma$ is generated by $N_1, N_2\ldots, N_k$.
\end{proof}

We present applications of Proposition \ref{prop-ex-c-a}.

\begin{prop}
If $\Gamma_1,\ldots, \Gamma_k$ are non-amenable discrete groups with $k\geq 2$ and if $\Gamma_1$ contains a subgroup isomorphic to a non-abelian free group, then the direct product $\Gamma_1\times \cdots \times \Gamma_k$ belongs to Class $\mathscr{A}$.
\end{prop}

\begin{proof}
Setting $\Gamma =\Gamma_1\times \cdots \times \Gamma_k$, we identify each $\Gamma_i$ with a subgroup of $\Gamma$.
Let $M^1$ and $M^2$ be infinite cyclic subgroups of $\Gamma_1$ which generate the non-abelian free group of rank two.
We put $N_1^1=M^1$ and $N_1^2=M^2$ and define $N_1$ as the group generated by $M^1$ and $M^2$.
For each $i=2,\ldots, k$, we put $N_i=\Gamma_i$ and put $N_{k+1}=\Gamma_1$.
Applying Proposition \ref{prop-ex-c-a}, we see that $\Gamma$ belongs to $\mathscr{A}$.
\end{proof}

\begin{prop}\label{prop-c-a-mcg}
Let $S$ be a surface of genus $g$ with $p$ boundary components. 
If $3g+p\geq 7$ and $(g, p)\neq (2, 1)$, then any finite index subgroup of $\mod^*(S)$ belongs to Class $\mathscr{A}$.
\end{prop}

\begin{proof}
We first note that if $a$ and $b$ are elements of $V(S)$ with $I(a, b)\neq 0$, then the squares of the Dehn twists about $a$ and $b$ generate the non-abelian free group of rank two (see Theorem 3.2 in \cite{hamidi}).
As mentioned right after Proposition \ref{prop-c-f-mcg}, if $g\geq 1$, then the Dehn twists about curves in the family
\[F=\{ \alpha_i\}_{i=1}^g\cup \{ \alpha_{j, j+1}\}_{j=1}^{g-1}\cup \{ \beta_k\}_{k=1}^{p-1}\cup \{ \beta, \gamma \}\]
described in Figure \ref{fig-seq} (a) generate $\pmod(S)$, where $\gamma$ is excluded if $g=1$.
When $g\geq 1$, we aim to construct a sequence of pairs of curves in $S$,
\[\{ \delta_1^1, \delta_1^2\},\ \{ \delta_2^1, \delta_2^2\},\ldots,\ \{ \delta_l^1, \delta_l^2\}\]
such that
\begin{enumerate}
\item[(1)] $I(\delta_i^1, \delta_i^2)\neq 0$ for any $i=1,\ldots, l$;
\item[(2)] $I(\delta_i^j, \delta_{i+1}^k)=0$ for any $i=1,\ldots, l-1$ and any $j, k=1, 2$; and
\item[(3)] $F\subset \{\, \delta_i^j\mid i=1,\ldots, l,\ j=1, 2\,\}$.
\end{enumerate}
Once such a sequence is found, for each $i=1,\ldots, l$, we define $N_i$ as the subgroup of $\pmod(S)$ generated by the Dehn twists about $\delta_i^1$ and $\delta_i^2$.
Condition (1) implies that each $N_i$ contains a non-abelian free subgroup.
Condition (2) implies that for each $i=1,\ldots, l-1$, any element of $N_i$ and any element of $N_{i+1}$ commute.
Condition (3) implies that $N_1,\ldots, N_l$ generate $\pmod(S)$.
By Proposition \ref{prop-ex-c-a}, $\pmod(S)$ belongs to $\mathscr{A}$.
It follows from Proposition \ref{prop-per-c-a} (ii) that any finite index subgroup of $\mod^*(S)$ belongs to $\mathscr{A}$.

\medskip

\noindent {\it Case $g\geq 3$}.
We define a curve $\alpha'$ in $S$ as in Figure \ref{fig-seq} (a) and put $z=\{ \alpha_g, \alpha'\}$.
The sequence of pairs of curves in $S$,
\begin{align*}
&\{ \alpha_1, \beta_1\},\ z,\ \{ \alpha_1, \beta_2\},\ldots,\ z,\ \{ \alpha_1, \beta_i\},\ldots,\ z,\ \{ \alpha_1, \beta_{p-1}\},\ z,\ \{ \alpha_1, \beta \},\ \{ \alpha_2, \gamma \},\ z,\\
&\{ \alpha_{1, 2}, \alpha_2\},\ z,\ \{ \alpha_{2, 3}, \alpha_3\},\ldots,\ z,\ \{ \alpha_{j, j+1}, \alpha_{j+1}\},\ldots,\ z,\ \{ \alpha_{g-2, g-1}, \alpha_{g-1}\},\ z,\\
&\{ \alpha_1, \beta \},\ \{ \alpha_{g-1, g}, \alpha_g\},
\end{align*}
then satisfies conditions (1)--(3).

\medskip

\noindent {\it Case $g=2$ and $p\geq 2$}.
We put $z=\{ \alpha_2, \gamma \}$.
Let $\beta'$ be the curve described in Figure \ref{fig-seq} (a), which is essential because $p\geq 2$.
The sequence of pairs of curves in $S$,
\begin{align*}
&\{ \alpha_1, \beta_1\},\ z,\ \{ \alpha_1, \beta_2\},\ldots,\ z,\ \{ \alpha_1, \beta_i\},\ldots,\ z,\ \{ \alpha_1, \beta_{p-1}\},\ z,\ \{ \alpha_1, \beta \},\ z,\ \{ \beta', \beta_1\},\\
&\{ \alpha_{1, 2}, \alpha_2\},
\end{align*}
then satisfies conditions (1)--(3).

\medskip

\noindent {\it Case $g=1$ and $p\geq 4$}.
We put $\beta_p=\beta$.
For each $i=1,\ldots, p-1$, choose two curves $\gamma_i^1, \gamma_i^2\in V(S)$ with
\[I(\gamma_i^1, \gamma_i^2)\neq 0,\quad I(\gamma_i^j, \alpha_1)=I(\gamma_i^j, \beta_i)=I(\gamma_i^j, \beta_{i+1})=0\]
for any $j=1, 2$.
They exist because $p\geq 4$.
The sequence of pairs of curves in $S$,
\begin{align*}
&\{ \alpha_1, \beta_1\},\ \{ \gamma_1^1, \gamma_1^2\},\ \{ \alpha_1, \beta_2\},\ \{ \gamma_2^1, \gamma_2^2\},\ldots,\ \{ \alpha_1, \beta_i\},\ \{ \gamma_i^1, \gamma_i^2\},\\
&\ldots,\ \{ \alpha_1, \beta_{p-1}\},\ \{ \gamma_{p-1}^1, \gamma_{p-1}^2\},\ \{ \alpha_1, \beta_p\},
\end{align*}
then satisfies conditions (1)--(3).

\medskip

\noindent {\it Case $g=0$ and $p\geq 7$}.
Label boundary components of $S$ as $\partial_1,\ldots, \partial_p$, where indices are regarded as elements of $\mathbb{Z}/p\mathbb{Z}$.
For each $i\in \mathbb{Z}/p\mathbb{Z}$, we define a curve $\alpha_i$ as described in Figure \ref{fig-seq} (b).
We put $F=\{\, \alpha_i \mid i\in \mathbb{Z}/p\mathbb{Z}\,\}$.
As mentioned right after Proposition \ref{prop-c-f-mcg}, the half twists about curves in $F$ generate $\mod(S)$.
The sequence of pairs of curves in $S$,
\begin{align*}
&\{ \alpha_1, \alpha_2\},\ \{ \alpha_5, \alpha_6\}, \{ \alpha_2, \alpha_3\},\ \{ \alpha_6, \alpha_7\},\ldots,\ \{ \alpha_i, \alpha_{i+1}\},\ \{ \alpha_{i+3}, \alpha_{i+4}\},\\
&\ldots,\ \{ \alpha_{p-1}, \alpha_p \},\ \{ \alpha_{p+2}, \alpha_{p+3}\},
\end{align*}
then satisfies conditions (1)--(3).
A verbatim argument shows that any finite index subgroup of $\mod^*(S)$ belongs to $\mathscr{A}$.
\end{proof}

One can also prove that if either $3g+p<7$ or $(g, p)=(2, 1)$, then for any finite index subgroup $\Gamma$ of $\mod^*(S)$, there exists no collection of at least two subgroups of $\Gamma$, $N_1, N_2,\ldots, N_k$, satisfying all the assumptions in Proposition \ref{prop-ex-c-a}.
We do not present a proof of this fact because it will not be used in the sequel.


\section{Coupling rigidity}\label{sec-coup-rigid}

We begin with the definition of measure equivalence (see \cite{furman-mer} for basics).

\begin{defn}[\ci{0.5.E}{gromov-as-inv}]\label{defn-me}
Two discrete groups $\Gamma$ and $\Lambda$ are said to be {\it measure equivalent (ME)} if one has a standard Borel space $(\Sigma, m)$ with a $\sigma$-finite positive measure and a measure-preserving action of $\Gamma \times \Lambda$ on $(\Sigma, m)$ such that there exist Borel subsets $X, Y\subset \Sigma$ of finite measure satisfying the equality
\[\Sigma =\bigsqcup_{\gamma \in \Gamma}\gamma Y=\bigsqcup_{\lambda \in \Lambda}\lambda X\]
up to $m$-null sets. 
In this case, the space $(\Sigma, m)$ equipped with the action of $\Gamma \times \Lambda$ is called a {\it coupling} of $\Gamma$ and $\Lambda$. 
When $\Gamma$ and $\Lambda$ are ME, we write $\Gamma \sim_{\rm ME}\Lambda$.
\end{defn}

Let $\Gamma$ be a discrete group.
We mean by a {\it self-coupling} of $\Gamma$ a coupling of $\Gamma$ with itself. 
Given a homomorphism from $\Gamma$ into a standard Borel group $G$, we define a rigidity property of self-couplings of $\Gamma$, called coupling rigidity. 
Once this property is obtained, for any discrete group $\Lambda$ with $\Gamma \sim_{\rm ME}\Lambda$, one gets a homomorphism from $\Lambda$ into $G$ thanks to a theorem due to Furman \cite{furman-mer} (see Theorem 3.5 in \cite{kida-ama} for a precise statement). 
This leads to understanding of structure of $\Lambda$.

\begin{defn}\label{defn-coup-rigid}
Let $\Gamma$ be a discrete group, $G$ a standard Borel group and $\pi \colon \Gamma \rightarrow G$ a homomorphism.
We say that $\Gamma$ is {\it coupling rigid} with respect to the pair $(G, \pi)$ if
\begin{itemize}
\item for any self-coupling $\Sigma$ of $\Gamma$, there exists an almost $(\Gamma \times \Gamma)$-equivariant Borel map $\Phi \colon \Sigma \rightarrow G$, where the action of $\Gamma \times \Gamma$ on $G$ is defined by the formula $(\gamma_1, \gamma_2)g=\gamma_1g\gamma_2^{-1}$ for any $\gamma_1, \gamma_2\in \Gamma$ and $g\in G$; and
\item the Dirac measure on the neutral element of $G$ is the only probability measure on $G$ that is invariant under conjugation by any element of $\pi(\Gamma)$.
\end{itemize}
When $\pi$ is understood from the context, $\Gamma$ is simply said to be coupling rigid with respect to $G$. 
\end{defn}

It is known that if $S$ is a surface of genus $g$ with $p$ boundary components satisfying $3g+p\geq 5$ and $(g, p)\neq (1, 2), (2, 0)$, then any finite index subgroup $\Gamma$ of $\mod^*(S)$ is coupling rigid with respect to the pair $(\mod^*(S), \imath)$, where $\imath$ is the inclusion of $\Gamma$ into $\mod^*(S)$ (see Corollary 5.9 of \cite{kida-mer}).

As a consequence of argument so far, we obtain the following coupling rigidity of amalgamated free products of discrete groups in Classes $\mathscr{F}$ and $\mathscr{A}$.

\begin{thm}\label{thm-coup-rigid}
Let $\Gamma =\Gamma_1\ast_A\Gamma_2$ be an amalgamated free product of discrete groups such that $A$ is infinite and proper in $\Gamma_i$ for each $i=1, 2$; and either of the following conditions (a) or (b) holds:
\begin{enumerate}
\item[(a)] For each $i=1, 2$, we have $\Gamma_i\in \mathscr{F}$, and the group $A\cap \gamma A\gamma^{-1}$ is finite for each $\gamma \in \Gamma_i\setminus A$.
\item[(b)] For each $i=1, 2$, we have $\Gamma_i\in \mathscr{A}$, the group $A\cap \gamma A\gamma^{-1}$ is amenable for each $\gamma \in \Gamma_i\setminus A$, and the equality ${\rm LQN}_{\Gamma_i}(A)=A$ holds.
\end{enumerate} 
Then $\Gamma$ is coupling rigid with respect to the pair $(\aut^*(T), \imath)$, where $T$ is the Bass-Serre tree associated with the decomposition $\Gamma =\Gamma_1\ast_A\Gamma_2$ and $\imath \colon \Gamma \rightarrow \aut^*(T)$ is the homomorphism arising from the action of $\Gamma$ on $T$.
\end{thm}

Let us recall the Bass-Serre tree $T$ associated with an amalgamated free product $\Gamma =\Gamma_1\ast_A\Gamma_2$. 
We refer to \cite{serre} for details. 
The sets of vertices and edges of $T$ are defined to be
\[V(T)=\Gamma /\Gamma_1\sqcup \Gamma /\Gamma_2,\quad E(T)=\Gamma /A,\]
respectively. 
For each $\gamma \in \Gamma$, the edge corresponding to the coset $\gamma A$ joins the two vertices corresponding to the cosets $\gamma \Gamma_1$ and $\gamma \Gamma_2$. 
The group $\Gamma$ then acts on $T$ as simplicial automorphisms by left multiplication.

\begin{proof}[Proof of Theorem \ref{thm-coup-rigid}]
Most of the proof is a verbatim translation of that of Theorem 4.4 in \cite{kida-ama}.
Let $\Sigma$ be a self-coupling of $\Gamma$. 
To distinguish the two actions of $\Gamma$ on $\Sigma$, we put $L(\Gamma)=\Gamma \times \{ e\}$ and $R(\Gamma)=\{ e\} \times \Gamma$.
Let $X\subset \Sigma$ be a fundamental domain for the action of $R(\Gamma)$ on $\Sigma$. 
We then have the natural action of $L(\Gamma)$ on $X$ and the ME cocycle $\alpha \colon \Gamma \times X\rightarrow \Gamma$ defined by the condition $(\gamma, \alpha(\gamma, x))x\in X$ for each $\gamma \in \Gamma$ and a.e.\ $x\in X$. 
We denote by $\cal{G}$ the groupoid associated with the action $L(\Gamma)\c X$. 
For each $v\in V(T)$, we define a subgroupoid $\cal{G}_v$ of $\cal{G}$ by setting $\cal{G}_v=\{\, (\gamma, x)\in \cal{G}\mid \gamma v=v\,\}$.

Pick $v\in V(T)$. 
The first step for the proof of Theorem 4.4 in \cite{kida-ama} is to find a $(\cal{G}_v, \alpha)$-invariant Borel map $\varphi_v\colon X\rightarrow V(T)$. 
In that proof, we found it by using property (T) of $\cal{G}_v$. 
In the present setting, the stabilizer in $\Gamma$ of any two distinct edges of $T$ is contained in a conjugate of $A\cap \gamma A\gamma^{-1}$ for some $\gamma \in (\Gamma_1\setminus A)\cup (\Gamma_2\setminus A)$, which is finite (resp.\ amenable) if we assume condition (a) (resp.\ (b)). 
The composition of the cocycle $\alpha \colon \Gamma \times X\rightarrow \Gamma$ and the homomorphism $\imath \colon \Gamma \rightarrow \aut^*(T)$ thus satisfies condition (I) in Definition \ref{defn-c-f} (resp.\ condition (Ia) in Definition \ref{defn-c-a}). 
It satisfies condition (II) by Proposition \ref{prop-ame-subgrd} and Corollary \ref{cor-ame-action} and satisfies condition (III) because $\Gamma$ acts on $T$ without inversions. 
Since $\Gamma_v$ belongs to Class $\mathscr{F}$ (resp.\ $\mathscr{A}$), there exists a $(\cal{G}_v, \alpha)$-invariant Borel map from $X$ into $V(T)$.

The rest of the proof is obtained by following the same argument as in the proof of the cited theorem.
\end{proof}

\begin{rem}
Let $\Lambda$ be a discrete group ME to the group $\Gamma$ in Theorem \ref{thm-coup-rigid}. 
As a consequence of that theorem, by the argument in Section 5.1 of \cite{kida-ama}, we can find a subgroup $\Lambda_+$ of $\Lambda$ with its index at most two and a simplicial action of $\Lambda_+$ on $T$ without inversions.
Moreover, for each simplex $s$ of $T$, the stabilizer of $s$ in $\Gamma$ and that of $s$ in $\Lambda_+$ are ME by Lemma 5.2 in \cite{kida-ama}. 
We refer to Sections 5.2 and 5.4 in \cite{kida-ama} for results on the quotient graph for the action of $\Lambda_+$ on $T$.
\end{rem}

The following corollary is immediately obtained by combining Theorem \ref{thm-coup-rigid} with Corollary 6.5 of \cite{kida-ama}. 
Let us say that a group $G$ is {\it ICC} if any conjugacy class in $G$ other than $\{ e\}$ consists of infinitely many elements.

\begin{cor}\label{cor-oer}
Let $\Gamma =\Gamma_1\ast_A\Gamma_2$ be the amalgamated free product in Theorem \ref{thm-coup-rigid}. 
We now assume that for each $i=1, 2$, there exist a discrete group $G_i$ and an injective homomorphism $\pi_i\colon \Gamma_i\rightarrow G_i$ such that $\Gamma_i$ is coupling rigid with respect to the pair $(G_i, \pi_i)$.

Let $\Lambda$ be a discrete group and suppose that two ergodic f.f.m.p.\ actions $\Gamma \c (X, \mu)$ and $\Lambda \c (Y, \nu)$ are WOE. 
We furthermore suppose the following two conditions:
\begin{itemize}
\item If we denote by $E\subset X$ and $F\subset Y$ the domain and range of the Borel isomorphism, respectively, that gives the WOE between the actions $\Gamma \c (X, \mu)$ and $\Lambda \c (Y, \nu)$, then we have $\mu(E)/\mu(X)\leq \nu(F)/\nu(Y)$; and
\item Either the action $A\c (X, \mu)$ is aperiodic or both of the actions $\Gamma_1\c (X, \mu)$ and $\Gamma_2\c (X, \mu)$ are aperiodic, the action $A\c (X, \mu)$ is ergodic, and $A$ is ICC.
\end{itemize}
Then we have $\mu(E)/\mu(X)=\nu(F)/\nu(Y)$, and the cocycle $\alpha \colon \Gamma \times X\rightarrow \Lambda$ associated with the WOE is cohomologous to the cocycle arising from an isomorphism from $\Gamma$ onto $\Lambda$. 
In particular, the two actions $\Gamma \c (X, \mu)$ and $\Lambda \c (Y, \nu)$ are conjugate.
\end{cor}

Theorem \ref{thm-oer} then follows from Propositions \ref{prop-c-f-mcg} and \ref{prop-c-a-mcg} and the coupling rigidity of mapping class groups proved in Corollary 5.9 of \cite{kida-mer}.


\section{Subgroups of mapping class groups}\label{sec-mcg}

In this section, we present three kinds of subgroups $A$ of the mapping class group $\Delta$ such that the group $\Delta \ast_A\Delta$ satisfies the assumption in Theorem \ref{thm-coup-rigid}. 
In Section \ref{subsec-pa} (resp.\ Sections \ref{subsec-pants} and \ref{subsec-disk}), we present a subgroup $A$ for which $\Delta \ast_A\Delta$ satisfies condition (a) (resp.\ (b)) in that theorem. 
In Section \ref{subsec-rem}, we make a brief comment on difficulty of finding a subgroup $A$ with $\Delta \ast_A\Delta$ ME rigid. 
Throughout this section, we denote by $S$ a surface of genus $g$ with $p$ boundary components unless otherwise stated.

\subsection{Stabilizers of pseudo-Anosov foliations}\label{subsec-pa}

In this subsection, we assume $3g+p\geq 5$. Let $\cal{PMF}=\cal{PMF}(S)$ denote the Thurston boundary for $S$. 
This is an ideal boundary of the Teichm\"uller space for $S$, and the group $\mod^*(S)$ acts on it continuously. 
The quotient of $\mod^*(S)$ by its center, which is of order at most two, acts on $\cal{PMF}$ faithfully. 
We recommend the reader consult \cite{iva-subgr} and \cite{iva-mcg} for fundamental facts on $\cal{PMF}$.

Pick a pseudo-Anosov element $f\in \mod(S)$. 
It is known that there exist exactly two points $F_{\pm}$ of $\cal{PMF}$ fixed by $f$. 
Let us call the set $\{ F_{\pm}\}$ the {\it pair of pseudo-Anosov foliations} for $f$. 
The following two facts are also known:
\begin{itemize}
\item The stabilizer of the pair $\{ F_{\pm}\}$ in $\mod^*(S)$ contains the cyclic subgroup generated by $f$ as a subgroup of finite index (see Lemma 5.10 in \cite{iva-subgr}).
\item After exchanging $F_+$ and $F_-$ if necessary, for any compact subset $K$ of $\cal{PMF}\setminus \{ F_-\}$ and any open neighborhood $U$ of $F_+$ in $\cal{PMF}$, we have $f^n(K)\subset U$ for any sufficiently large $n$ (see Theorem 3.5 in \cite{iva-subgr}).
\end{itemize}
The second fact implies that for any non-zero integer $n$, the set of fixed points in $\cal{PMF}$ for $f^n$ is equal to $\{ F_{\pm}\}$. 
One then obtains the following:

\begin{lem}\label{lem-pa}
Let $\Gamma$ be a finite index subgroup of $\mod^*(S)$. 
Pick a pseudo-Anosov element $f\in \mod(S)$ and define $A$ to be the stabilizer of the pair of pseudo-Anosov foliations for $f$ in $\Gamma$. 
Then $A$ is almost malnormal in $\Gamma$.
\end{lem}

\begin{proof}
Pick $\gamma \in \Gamma$ with $\gamma A\gamma^{-1}\cap A$ infinite. 
The first fact mentioned above implies that for some non-zero integer $n$, we have $f^n\in \gamma A\gamma^{-1}\cap A$. 
Since $\gamma A\gamma^{-1}$ fixes the pair $\{ \gamma F_{\pm}\}$, so does $f^n$. 
This implies $\{ \gamma F_{\pm}\} =\{ F_{\pm}\}$ and thus $\gamma \in A$.
\end{proof}

\subsection{Stabilizers of pants decompositions}\label{subsec-pants}

We assume $3g+p\geq 5$. 
Recall that $V(S)$ denotes the set of isotopy classes of essential simple closed curves in $S$. 
We define $\Sigma(S)$ to be the set of all non-empty finite subsets $\sigma$ of $V(S)$ with $I(\alpha, \beta)=0$ for any $\alpha, \beta \in \sigma$. 
For each $\sigma \in \Sigma(S)$, we denote by $S_{\sigma}$ the surface obtained by cutting $S$ along mutually disjoint representatives of elements of $\sigma$.
This ambiguity will be of no importance in the sequel. 
We put
\[\cal{P}(S)=\{\, \sigma \in \Sigma(S)\mid \left|\sigma \right|=3g+p-3\,\}\]
and call an element of $\cal{P}(S)$ a {\it pants decomposition} of $S$. 
We note that for each $\sigma \in \Sigma(S)$, we have $\left|\sigma \right|=3g+p-3$ if and only if each component of $S_{\sigma}$ is homeomorphic to a pair of pants, i.e., a surface of genus zero with three boundary components.

For each $\sigma \in \Sigma(S)$, we set
\[\stab(\sigma)=\{\, \gamma \in \mod^*(S)\mid \gamma \sigma =\sigma\,\}.\]
Assertion (i) in the following lemma is a consequence of Corollary 4.1.B in \cite{iva-mcg}. 
Assertion (ii) is proved in Proposition 5.1 of \cite{paris}.

\begin{lem}
The following assertions hold:
\begin{enumerate}
\item If $\sigma \in \cal{P}(S)$, then the group generated by all Dehn twists about curves in $\sigma$ is of finite index in $\stab(\sigma)$. In particular, $\stab(\sigma)$ is amenable.
\item If $\sigma, \tau \in \Sigma(S)$ satisfy $\tau \not\subset \sigma$, then the orbit for the action $\stab(\sigma)\c \Sigma(S)$ containing $\tau$ consists of infinitely many elements.
\end{enumerate}
\end{lem}

Assertion (ii) implies the equality
\[{\rm LQN}_{\mod^*(S)}(\stab(\sigma))=\stab(\sigma)\]
for each $\sigma \in \Sigma(S)$. 
Let $\Delta$ be a finite index subgroup of $\mod^*(S)$.
For each $\sigma \in \Sigma(S)$, we set $\Delta_{\sigma}=\stab(\sigma)\cap \Delta$.
It follows that for each $\sigma \in \cal{P}(S)$, the group $\Delta \ast_{\Delta_{\sigma}}\Delta$ fulfills condition (b) in Theorem \ref{thm-coup-rigid}.

\begin{rem}
For each surface $Q$, let $\cal{A}(Q)$ denote the subset of $\cal{PMF}(Q)$ consisting of all projective classes of pseudo-Anosov foliations on $Q$ associated to a pseudo-Anosov element of $\mod(Q)$.

For each $\sigma \in \Sigma(S)$, we define a subset $\cal{F}(\sigma)$ of $\cal{PMF}$ to be the set of all projective classes of measured foliations on $S$ defined as the sum
\[\sum_Qs_QF_Q+\sum_{\alpha \in \sigma}t_{\alpha}\alpha\]
with $F_Q\in \cal{A}(Q)$, $s_Q>0$ and $t_{\alpha}\geq 0$, where $Q$ runs through all components of $S_{\sigma}$ that are not homeomorphic to a pair of pants. 
The reader should consult Section 2 of \cite{iva-subgr} for measured foliations and their sums. 
For each $F\in \cal{F}(\sigma)$, we denote by $\stab(F)$ the stabilizer of $F$ in $\mod^*(S)$. 
Along the same argument as above, one can verify that for each $F\in \cal{F}(\sigma)$, the group $\stab(F)$ is amenable and we have
\[{\rm LQN}_{\mod^*(S)}(\stab(F))=\stab(F).\]
\end{rem}

\subsection{Stabilizers of Teichm\"uller disks}\label{subsec-disk}

Throughout this subsection, we assume $S$ to be a closed surface of genus at least two. 
We first summarize known facts on Teichm\"uller disks, whose details and related topics are found in \cite{earle-gardiner} and \cite{masur-tab}. 
Let $\cal{T}$ denote the Teichm\"uller space for $S$. 
Pick a point $x\in \cal{T}$ and let $X$ be a Riemann surface associated with $x$. 
For each holomorphic quadratic differential $\varphi$ on $X$, one can construct a totally geodesic subspace $D$ of $\cal{T}$ containing $x$ and isometric to the Poincar\'e disk, with respect to the Teichm\"uller metric on $\cal{T}$. 
This $D$ is called the {\it Teichm\"uller disk} determined by $\varphi$. 
We now focus on the stabilizer $\stab(D)$ of $D$ for the action of $\mod^*(S)$ on $\cal{T}$. 
We have the homomorphism
\[\eta \colon \stab(D)\rightarrow \isom(D)\simeq PGL(2, \mathbb{R}),\]
where $\isom(D)$ is the isometry group of $D$. 
Since the action of $\mod^*(S)$ on $\cal{T}$ is properly discontinuous (see Chapter 8 in \cite{gardiner}), $\ker \eta$ is finite and $\eta(\stab(D))$ is discrete in $PGL(2, \mathbb{R})$. 
The group $\eta(\stab(D)\cap \mod(S))$ is often called the {\it Veech group} for $D$. 
Let us recall fundamental properties of Teichm\"uller disks.

\begin{thm}[\ci{Theorem 6}{kra}]\label{thm-disk1}
Let $D$ be a Teichm\"uller disk in $\cal{T}$, and pick an element $f\in \stab(D)\cap \mod(S)$. Then the following assertions hold:
\begin{enumerate}
\item If $\eta(f)$ is elliptic or the identity, then $f$ is of finite order.
\item If $\eta(f)$ is parabolic, then $f$ is reducible.
\item If $\eta(f)$ is hyperbolic, then $f$ is pseudo-Anosov.
\end{enumerate}
\end{thm}

\begin{thm}\label{thm-disk2}
The following assertions hold:
\begin{enumerate}
\item For any two Teichm\"uller disks $D_1$ and $D_2$ in $\cal{T}$, we have $\left|D_1\cap D_2\right|\geq 2$ if and only if $D_1=D_2$.
\item For any pseudo-Anosov element $f\in \mod(S)$, there exists exactly one Teichm\"uller disk $D$ with $f\in \stab(D)$.
\end{enumerate}
\end{thm}

In the last theorem, assertion (i) is proved in Lemma 3.1 of \cite{marden-masur}, and assertion (ii) follows from assertion (i) and the fact that for each pseudo-Anosov element $f$, there exists exactly one geodesic in $\cal{T}$ stabilized by $f$. 
This is proved by combining results in Sections 6 and 9 of \cite{bers} and Theorem 7.3.A of \cite{iva-mcg}. 
Using these theorems, we obtain the following:

\begin{lem}\label{lem-teich}
Let $\Gamma$ be a finite index subgroup of $\mod^*(S)$, and pick a pseudo-Anosov element $f\in \mod(S)$. 
Let $D$ be the Teichm\"uller disk with $f\in \stab(D)$, and set $A=\stab(D)\cap \Gamma$. 
Then the equality ${\rm LQN}_{\Gamma}(A)=A$ holds and the group $\gamma A\gamma^{-1}\cap A$ is amenable for each $\gamma \in \Gamma \setminus A$.
\end{lem}

\begin{proof}
Pick $\gamma \in \Gamma \setminus A$ and put $B=\gamma A\gamma^{-1}\cap A$.
The group $B$ is equal to $\stab(\gamma D)\cap \stab(D)$.
If the index $[A :B]$ were finite, then for some positive integer $n$, $f^n$ would lie in $B$. 
Since $\gamma D\neq D$, this contradicts Theorem \ref{thm-disk2} (ii). 
We thus obtain the equality ${\rm LQN}_{\Gamma}(A)=A$. 
Theorem \ref{thm-disk1} and Theorem \ref{thm-disk2} (ii) imply that each element of $B\cap \mod(S)$ acts on $D$ as an elliptic or parabolic isometry. 
It then follows from Theorem 2.4.4 of \cite{katok} that $\eta(B\cap \mod(S))$ is an elementary subgroup of $PSL(2, \mathbb{R})$ and is thus amenable.
\end{proof}

\subsection{Concluding remarks}\label{subsec-rem}

If one can find either
\begin{enumerate}
\item[(1)] an infinite, ICC, proper and almost malnormal subgroup $A$ of $\mod^*(S)$ when $3g+p\geq 5$ and $(g, p)\neq (1, 2), (2, 0)$; or
\item[(2)] an ICC and proper subgroup $A$ of $\mod^*(S)$ such that we have the equality ${\rm LQN}_{\mod^*(S)}(A)=A$ and the group $\gamma A\gamma^{-1}\cap A$ is amenable for each $\gamma \in \mod^*(S)\setminus A$ when $3g+p\geq 7$ and $(g, p)\neq (2, 1)$,
\end{enumerate}
then the following rigidity result is obtained as a consequence of Theorem 7.10 of \cite{kida-ama} and Theorem \ref{thm-coup-rigid}: For each finite index subgroup $\Delta$ of $\mod^*(S)$, the amalgamated free product $\Gamma =\Delta \ast_{A\cap \Delta}\Delta$ is coupling rigid with respect to the abstract commensurator of $\Gamma$. 
In particular, $\Gamma$ is {\it ME rigid}, that is, any discrete group ME to $\Gamma$ is virtually isomorphic to $\Gamma$. 
We can also obtain a rigidity property of every ergodic f.f.m.p.\ action of $\Gamma$.

It however seems difficult to find a subgroup $A$ of $\mod^*(S)$ satisfying the above condition (1) or (2). 
The following proposition shows that it is impossible to find an $A$ satisfying condition (1) in the class of subgroups of $\mod^*(S)$ containing a reducible element of infinite order. 
We say that an element $f\in \mod^*(S)$ is {\it reducible} if $f$ fixes an element of $\Sigma(S)$.

\begin{prop}
Let $S$ be a surface with $3g+p\geq 5$, and let $\Delta$ be a subgroup of $\mod^*(S)$ containing a reducible element of infinite order. 
If $\Delta$ is almost malnormal in $\mod^*(S)$, then the equality $\Delta =\mod^*(S)$ holds.
\end{prop}

\begin{proof}
Recall that we have the equality
\[ft_{\alpha}f^{-1}=t_{f\alpha}, \quad \forall f\in \mod(S),\ \forall \alpha \in V(S),\]
where $t_{\alpha}\in \mod(S)$ denotes the Dehn twist about $\alpha \in V(S)$ (see Lemma 4.1.C in \cite{iva-mcg}). 
Let $f\in \Delta$ be a reducible element of infinite order and choose $\alpha_0\in V(S)$ and a positive integer $n$ with $f^n\alpha_0=\alpha_0$ and $f^n\in \mod(S)$. 
The above equality shows that $t_{\alpha_0}$ and $f^n$ commute. 
It follows that $t_{\alpha_0}$ belongs to $\Delta$ since $\Delta$ is almost malnormal in $\mod^*(S)$.

As in the proof of Proposition \ref{prop-c-f-mcg}, if $g\geq 1$, then we can find elements $\alpha_1,\ldots, \alpha_n$ of $V(S)$ such that
\begin{itemize}
\item $I(\alpha_i, \alpha_{i+1})=0$ for each $i=0, 1,\ldots, n-1$; and
\item $t_{\alpha_0}, t_{\alpha_1},\ldots, t_{\alpha_n}$ generate $\pmod(S)$.
\end{itemize}
Since $t_{\alpha_i}$ and $t_{\alpha_{i+1}}$ commute for each $i$, we can conclude that $\Delta$ contains $\pmod(S)$.
The equality $\Delta =\mod^*(S)$ then follows.

A similar argument can be applied to proving the equality $\Delta =\mod^*(S)$ in the case of $g=0$.
\end{proof}

The following results on Teichm\"uller disks show us difficulty of finding a subgroup $A$ of $\mod^*(S)$ satisfying condition (2). 
As discussed in Section 1.8 of \cite{masur-tab}, there is a one-to-one correspondence between integrable meromorphic quadratic differentials on $S$ and flat structures on $S$ with cone singularities. 
The latter viewpoint is often helpful in understanding the group $\stab(D)$ for a Teichm\"uller disk $D$.
Among other things, examples of $D$ with $\eta(\stab(D))$ a lattice in $PGL(2, \mathbb{R})$ are found in \cite{veech} (see also Section 5 in \cite{masur-tab}). 
Notice that any lattice in $PGL(2, \mathbb{R})$ is ICC. 
Many such examples of $D$ are obtained by using billiards in rational polygons. 
We refer to Section 1 in \cite{masur-tab} for construction of the flat structure on $S$ associated to the billiard in a rational polygon. 
Proposition 3 in \cite{gutkin} however shows that if $D$ is the Teichm\"uller disk associated to such a billiard, then $\ker \eta$ is non-trivial.
It follows that $\stab(D)$ cannot be ICC since $\ker \eta$ is finite.


\end{document}